\documentclass[reqno]{amsart}

\usepackage{article-preamble}

\hypersetup{colorlinks=true,
    linkcolor=blue,
    citecolor=blue,
    pdftitle={Strong exceptional parameters for the dimension of nonlinear slices},
    pdfauthor={Ryan E.~G.~Bushling}}

\begin{document}

\setlength{\abovedisplayskip}{3pt}
\setlength{\belowdisplayskip}{3pt}
\setlength{\abovedisplayshortskip}{1.5pt}
\setlength{\belowdisplayshortskip}{1.5pt}

\title[Strong exceptional parameters for the dimension of nonlinear slices]{Strong exceptional parameters for \\ the dimension of nonlinear slices}
\author{Ryan E.~G.~Bushling}
\address{Department of Stochastics \\ Budapest University of Technology and Economics, 1. Egry J\'{o}zsef Street \\ Budapest, Hungary 1111}
\email{rbushling@edu.bme.hu}

\begin{abstract}
    Let $1 \leq m < s \leq n$ and let $A \subseteq \mathbb{R}^n$ be a Borel set of with $s$-dimensional Hausdorff measure $\mathcal{H}^s(A) > 0$. The classical Marstrand slicing theorem states that, for almost every $m$-dimensional subspace $V \subset \R^n$, there is a positive-measure set of $\mathbf{x} \in V$ such that $\mathbf{x} + V^\perp$ intersects $A$ in a set of Hausdorff dimension $s-m$. We prove a strong and quantitative version of Marstrand's slicing theorem in the Peres--Schlag framework. In particular, if $(\Pi_{\bm{\lambda}}: \Omega \to \mathbb{R}^m)_{\bm{\lambda} \in U}$ is a family of generalized projections that satisfies the transversality and strong regularity conditions of degree $0$, then for every $A \subseteq \Omega$ with $\mathcal{H}^s(A) > 0$, the set of $\bm{\lambda}$ in the parameter space $U \subseteq \mathbb{R}^N$ such that $\dim\!\big(A \cap \Pi_{\bm{\lambda}}^{-1}(\mathbf{x})\big) < s-m$ for a.e.\! $\mathbf{x} \in \mathbb{R}^m$ has Hausdorff dimension at most $N + m - s$. If moreover $\mathcal{H}^s(A) < \infty$, then this exceptional set is universal for the subsets of $A$ with positive $s$-dimensional Hausdorff measure in the sense that this same collection of parameters contains the corresponding exceptional sets of all those subsets of $A$. When $(\Pi_{\bm{\lambda}})_{\bm{\lambda} \in U}$ is only transversal and strongly regular of some sufficiently small order $\beta > 0$, a similar conclusion holds modulo an error term of order $\beta^{1/3}$.
\end{abstract}

\maketitle

\section{Introduction} \label{s:intro}

\subsection{Main results} Our purpose is to build on a paper \cite{orponen2014slicing} of Orponen that sharpens Marstrand's slicing theorem and extends it to slices by fibers of generalized projections in the sense of Peres--Schlag \cite{peres2000smoothness}. The terminology of the Peres--Schlag paper, although standard, is also substantial, so we first state the main result in the classical setting of slices by affine subspaces of Euclidean space. Given integers $1 \leq m < n$, let $\Gr(n,m)$ denote the Grassmannian of $m$-dimensional subspaces of $\R^n$; and, for $V \in \Gr(n,m)$ let $\pi_{\scalebox{0.65}{$V$}}$ denote the orthogonal projection of $\R^n$ onto $V \cong \R^m$.

\begin{thm} \label{thm:ortho}
    Let $m < s < n$ and let $A \subseteq \R^n$ be a Borel set with $0 < \mathcal{H}^s(A) < \infty$. There exists a set $E \subset \Gr(n,m)$ with
    \begin{equation*}
        \hdim E \leq m(n-m) + m - s
    \end{equation*}
    such that, for every subspace $V \in \Gr(n,m) \setminus E$ and every Borel set $B \subseteq A$ with $\mathcal{H}^s(B) > 0$, the equation
    \begin{equation} \label{eq:ortho-fiber-bounds}
        \hdim\!\big(B \cap \pi_{\scalebox{0.65}{$V$}}^{-1}(\mathbf{x})\big) = s-m
    \end{equation}
    holds for a positive-$\mathcal{H}^m$-measure set of $\mathbf{x} \in V$.
\end{thm}

Here and throughout, $\dim$ denotes Hausdorff dimension and $\mathcal{H}^r$ denotes $r$-dimensional Hausdorff measure ($r \geq 0$) in a separable metric space.

Informally, for all $V \in \Gr(n,m)$ outside a small ``bad" set $E$, the intersections of $A$ with the orthogonal spaces $\mathbf{x} + V^\perp$ are large (relative to $s$) for many points $\mathbf{x} \in V$. Moreover, $E$ is \textit{universal} for $A$ in the sense that this same statement holds for all subsets $B \subseteq A$ of positive $\mathcal{H}^s$-measure: $E$ contains the exceptional set of directions $V \in \Gr(n,m)$ corresponding to \textit{every} such $B$. Following Falconer and Mattila \cite{falconer2016strong}, we accordingly refer to the estimate in Theorem \ref{thm:ortho} as ``strong."

The upper bound on the dimension of slices follows immediately from the well-known sublevel set inequality for the Hausdorff measure of Lipschitz maps (see, for example, \cite{mattila1995geometry} Theorem 7.7). In particular, since $\pi_{\scalebox{0.65}{$V$}}$ is Lipschitz for \textit{all} $V \in \Gr(n,m)$, the set $A \cap \pi_{\scalebox{0.65}{$V$}}^{-1}(\mathbf{x})$ (hence, every subset thereof) has finite $(s-m)$-dimensional Hausdorff measure for $\mathcal{H}^m$-a.e.\! $\mathbf{x} \in V$. \textit{A fortiori} there are \textit{no} $V \in \Gr(n,m)$ such that the upper bound in \eqref{eq:ortho-fiber-bounds} fails for $\mathcal{H}^m$-positively-many $\mathbf{x}$.

The lower bound $\hdim\!\big(B \cap \pi_{\scalebox{0.65}{$V$}}^{-1}(\mathbf{x})\big) \geq s-m$ in Theorem \ref{thm:ortho} is a special case of our main result, which is situated in the Peres--Schlag framework. In both the statement and the subsequent proofs, it may be efficacious to begin by considering only the case $\beta = 0$. See \S \ref{s:background} for a review of the definitions, and see Theorem \ref{thm:weak-main} for a simpler statement that forgoes the universality of the exceptional set.

\begin{thm} \label{thm:main}
    Let $U \subseteq \R^N$ be an open set and let $\big( \Pi_{\bm{\lambda}} \!: \Omega \to \R^m \big)_{\bm{\lambda} \in U}$ be a family of generalized projections that is transversal and strongly regular of degree $\beta \in [0,1)$. There exists a constant $a \geq 1$, depending only on $N$ and $m$, such that the following holds. If $m < s < N+m$, then for every Borel set $A \subseteq \Omega$ with $0 < \mathcal{H}^s(A) < \infty$, there exists a set $E \subset U$ with
    \begin{equation*}
        \hdim E \leq N + m - s + a \beta^{1/3}
    \end{equation*}
    such that, for every $\bm{\lambda} \in U \setminus E$ and every Borel set $B \subseteq A$ with $\mathcal{H}^s(B) > 0$, the inequality
    \begin{equation} \label{eq:fiber-dim-bounds}
        \hdim\!\big(B \cap \Pi_{\bm{\lambda}}^{-1}(\mathbf{x})\big) \geq (1-\beta) (s-m)
    \end{equation}
    holds for a positive-Lebesgue-measure set of $\mathbf{x} \in \R^m$.
\end{thm}

Note that the dimension bound $N + m - s + \eps(\beta)$ on $E$, where $\eps(\beta) \to 0$ as $\beta \to 0$, is precisely the form that the corresponding projection result takes (see Theorem \ref{thm:peres-schlag} below). This will be important in \S \ref{s:main-proof}.

As a side note, since the ambient dimension $\hdim \Omega$ plays no role and the transversality and regularity of $(\Pi_{\bm{\lambda}})_{\bm{\lambda} \in U}$ are preserved under restriction of the domain $\Omega$, it would suffice to take $A = \Omega$.

\vspace{-0.25cm} \subsection{Overview and outline} The methods of this paper are already well-established in geometric measure theory. J\"{a}rvenp\"{a}\"{a}--J\"{a}rvenp\"{a}\"{a}--Niemel\"{a} \cite{jarvenpaa2004transversal} proved a nonlinear slicing theorem for mappings between manifolds using the same basic definitions as Peres and Schlag, but their methods were not Fourier-analytic and did not yield dimension bounds on the exceptional sets. Orponen \cite{orponen2014slicing} then proved Theorem \ref{thm:main} for $m = N = 1$ and $B = A$, and the proof of this generalization follows his closely. Mattila \cite{mattila2015fourier} also proved the $B = A$ case of Theorem \ref{thm:ortho} by importing some of Orponen's machinery into higher dimensions (for which reason he attributes the result to Orponen), and the obstacles in extending their slicing result for orthogonal projections to generalized projections are more practical than theoretical. Once the details are worked out, the strong aspect of Theorem \ref{thm:main} follows from a simple adaptation of the argument in \cite{falconer2016strong}.

After laying out the preliminaries in \S \ref{s:background}, we state several technical lemmas in \S \ref{ss:lemmas} that adapt the Peres--Schlag framework to slices. Most of the proof of Theorem \ref{thm:main} is contained in that of Lemma \ref{lem:slicing-main} in \S \ref{ss:main-lemma}, and in turn the main hurdle in proving this lemma is verifying the hypotheses of Lemma \ref{lem:power-series}. Since the proof of Lemma \ref{lem:slicing-main} is quite lengthy as is, Lemma \ref{lem:technical} is set aside to facilitate the application of Lemma \ref{lem:power-series}, and following Orponen's example we relegate the cumbersome proof of Lemma \ref{lem:technical} to the \hyperlink{s:appendix}{Appendix}. With the main Lemma \ref{lem:slicing-main} in place, the proof of Theorem \ref{thm:main} is completed in two steps: first in \S \ref{ss:weak} without the ``universal" exceptional set, which is then introduced in \S \ref{ss:full} using the method of Falconer and Mattila, thereby completing the proof of the theorem. Finally, \S \ref{s:applications} presents an application of Theorem \ref{thm:main} to vertical slices in the Heisenberg groups $\bbh^n$.

\section{Background} \label{s:background}

Given parameters $w_1,\dots,w_k$ and quantities $a,b \in \R \cup \{\pm\infty\}$, we write $a \lesssim_{w_1,\dots,w_k} b$ if there exists a positive constant $c_{w_1,\dots,w_k}$ (depending only on those parameters) such that $a \leq c_{w_1,\dots,w_k} b$. Similarly, $a \gtrsim_{w_1,\dots,w_k} b$ means that $b \lesssim_{w_1,\dots,w_k} a$ and $a \sim_{w_1,\dots,w_k} b$ means that both $a \lesssim_{w_1,\dots,w_k} b$ and $a \gtrsim_{w_1,\dots,w_k} b$. For readability, some parameters (e.g., the dimension of the space in which we are working) may be suppressed from the notation when doing so poses no risk of error or confusion. Although this is less standard, we will also on occasion write $a \simeq_{w_1,\dots,w_k} b$ when $a = c_{w_1,\dots,w_k} b$. This is stronger than $a \sim_{w_1,\dots,w_k} b$, which only says that $c_{w_1,\dots,w_k} a \leq b \leq d_{w_1,\dots,w_k} a$.

\vspace{-0.15cm} \subsection{Preliminaries on Radon measures} Denote by $\mathcal{M}(\Omega)$ the space of all nonzero \textit{Radon measures} on a separable metric space $(\Omega,d)$, i.e., the positive cone of Borel regular (outer) measures that are finite on compact sets, inner regular on open sets, and outer regular on all sets, but not identically $0$ (see \cite{mattila1995geometry}). The $s$-dimensional Hausdorff measure $\mathcal{H}^s$ is Borel regular, and it is Radon when restricted to a fixed subset of $\Omega$ with finite $\mathcal{H}^s$-measure.

A measure $\mu \in \mathcal{M}(\Omega)$ is called \textit{$s$-Frostman} if
\begin{equation*}
    \mu(B(\bm{\omega},r)) \lesssim_\mu r^s \qquad \text{for all } \bm{\omega} \in \Omega \text{ and } r > 0.
\end{equation*}
The following result, whose importance is difficult to understate, will be used frequently without explicit mention. This version for general metric spaces is due to Howroyd \cite{howroyd1995dimension}.
\begin{thm}[Frostman's lemma]
    If $\Omega$ is a $\sigma$-compact metric space with $\mathcal{H}^s(\Omega)$ $> 0$, then $\Omega$ supports an $s$-Frostman measure.
\end{thm}
A related tool for describing the scaling of a measure $\mu \in \mathcal{M}(\Omega)$ is the \textit{$s$-energy}
\begin{equation} \label{eq:s-energy}
    I_s(\mu) := \iint_{\Omega \times \Omega} \frac{d(\mu \times \mu)(\bm{\omega},\bm{\zeta})}{d(\bm{\omega},\bm{\zeta})^s}.
\end{equation}
If $\mu$ is finite and $s$-Frostman, then it has finite $t$-energy for all $0 \leq t < s$. Conversely, if $\mu$ is finite and $I_s(\mu) < \infty$, then there exists a positive-measure subset $B \subseteq \Omega$ such that $\mu \restrict B$, the restriction of $\mu$ to $B$, is $s$-Frostman.

A useful formula for the $s$-energy of a finite Borel measure $\mu$ on $\R^k$ is
\begin{equation} \label{eq:sobolev-energy}
    I_s(\mu) = \gamma_{n,s} \int_{\R^k} |\mathbf{z}|^{s-k} \big| \medhat{\mu}(\mathbf{z}) \big|^2 \, d\mathbf{z}, \qquad 0 < s \leq k,
\end{equation}
where $\medhat{\mu} \!: \R^k \to \C$ is the Fourier transform of the measure $\mu$ and $\gamma_{n,s} > 0$ is a constant. However, the right-hand side of this equation can in fact converge for \textit{any} value of $s \in \R$ (simply choose $\mu$ to be a Schwartz function), whereas the integral in \eqref{eq:s-energy} diverges whenever $s > k$. In lieu of the right-hand side of \eqref{eq:sobolev-energy}, one often studies instead the quantity
\begin{equation*}
    \mathcal{I}_s(\mu) := \int_{\R^k} (1+|\mathbf{z}|)^{s-k} \big| \medhat{\mu}(\mathbf{z}) \big|^2 \, d\mathbf{z},
\end{equation*}
the \textit{$s$-Sobolev energy} of $\mu$, which has the practical advantage of a bounded integrand. As long as $s > 0$, the two are comparable. Accordingly, the number
\begin{equation*}
    \sdim \mu := \sup \left\{ t \in \R \!: \mathcal{I}_t(\mu) < \infty \right\}
\end{equation*}
is called the \textit{Sobolev dimension} of $\mu$, and it coincides with $s$ whenever $\mu$ is $s$-Frostman and $0 \leq s < n$ (but not necessarily when $s = n$). Clearly $\mathcal{I}_t(\mu) < \infty$ for all $t < 0$, so $\sdim \mu \geq 0$.

\vspace{-0.15cm} \subsection{Pushforwards and slices of measures} If $\Psi \!: \Omega \to \R^k$ is a continuous map and $\mu \in \mathcal{M}(\Omega)$, then there is a Radon measure $\Psi_\sharp \mu \in \mathcal{M}(\R^k)$---the \textit{pushforward} of $\mu$ by $\Psi$---defined on Borel sets $X \subseteq \R^k$ by $\Psi_\sharp \mu(X) := \mu(\Psi^{-1}(X))$. The pushforward measure satisfies $\spt \Psi_\sharp \mu = \Psi(\spt \mu)$ and, for all Borel functions $g \!: \R^k \to [0,\infty]$,
\begin{equation*}
    \int_{\R^k} g(\mathbf{z}) \, d(\Psi_\sharp \mu)(\mathbf{z}) = \int_\Omega (g \circ \Psi)(\bm{\omega}) \, d\mu(\bm{\omega}).
\end{equation*}
From this formula one can derive the powerful identity
\begin{equation} \label{eq:plancherel-identity}
    \int_{\R^k} g(\mathbf{z}) \big| \widehat{\Psi_\sharp \mu}(\mathbf{z}) \big|^2 \, d\mathbf{z} = \iint_{\Omega \times \Omega} \medhat{g}(\Psi(\bm{\omega}) - \Psi(\bm{\zeta})) \, d(\mu \times \mu)(\bm{\omega},\bm{\zeta})
\end{equation}
using Plancherel's theorem. Incidentally, \eqref{eq:sobolev-energy} is a special case of this.

There is a straightforward procedure for defining the slices of a measure by the fibers of a continuous map by means of its pushforward under said map. However, we will later relocate our analysis from $\Omega$ to a Euclidean space in such a way that the slices of a measure $\mu \in \mathcal{M}(\Omega)$ correspond to the slices of its ``surrogate" in $\R^n$ by affine subspaces. (Incidentally, this surrogate is itself a pushforward measure---the pushforward of $\mu$ by the map $\Psi_{\bm{\lambda}}$ defined in \S \ref{ss:gen-proj} below.) For that reason, we need only to understand the slices of measures on $\R^n$ by the fibers of orthogonal projections.

To that end, for each $V \in \Gr(m+n,m)$, let $\nu_{\scalebox{0.65}{$V$}} := \pi_{\scalebox{0.65}{$\+V$}\sharp} \nu \in \mathcal{M}(\R^m)$ be the pushforward of the measure $\nu \in \mathcal{M}(\R^{m+n})$ by $\pi_{\scalebox{0.65}{$V$}}$, where we identify $V$ with $\R^m$ using local coordinates on $\Gr(m+n,m)$. For $\mathbf{x} \in \R^m$, we define the \textit{slice of $\nu$ by $\pi_{\scalebox{0.65}{$V$}}^{-1}(\mathbf{x})$}, denoted $\nu_{\scalebox{0.65}{$\+V$}\-,\mathbf{x}}$, via its action on functions $f \in C_c^+(\R^{m+n})$:
\begin{equation*}
    \int_{\R^{m+n}} f(\mathbf{z}) \, d\nu_{\scalebox{0.65}{$\+V$}\-,\mathbf{x}}(\mathbf{z}) := \lim_{\delta \to 0} \frac{1}{\mathcal{L}^m(B(\mathbf{x},\delta))} \int_{\pi_{\scalebox{0.5}{$V$}}^{-1}(B(\mathbf{x},\delta))} f(\mathbf{w}) \, d\nu(\mathbf{w}),
\end{equation*}
provided this limit exists. This is precisely the Radon--Nikodym derivative $D(\pi_{\scalebox{0.65}{$V$}\sharp}\nu_f,\mathbf{x})$ of the pushforward by $\pi_{\scalebox{0.65}{$V$}}$ of the Radon measure $\nu_f(A) := \int_A f(\mathbf{z}) \+ d\nu(\mathbf{z})$ on $\R^{m+n}$, so for a \textit{fixed} $f \in C_c^+(\R^{m+n})$, the limit exists for $\mathcal{L}^m$-a.e.\! $\mathbf{x} \in \R^m$. A standard argument from the separability of $C_c(\R^{m+n})$ then implies the existence of this limit for \textit{all} $f \in C_c^+(\R^{m+n})$, so the sliced measure $\nu_{\scalebox{0.65}{$\+V$}\-,\mathbf{x}}$ exists as a positive linear functional on $C_c^+(\R^{m+n})$ for almost every $\mathbf{x}$. By the Riesz representation theorem, it is an honest Radon measure.

\vspace{-0.15cm} \subsection{Bessel functions and spherical harmonics} Special functions make a brief appearance in the proof of the main lemma. All the facts utilized are ``well-known'' and they are unimportant in the broader scope of this chapter, so we only record only the very basics here. See \cite{muller2012analysis} for a full exposition.

For each $\nu > -\tfrac{1}{2}$, let $J_\nu \!: [0,\infty) \to \R$ be the \textit{Bessel function of the first kind of order $\nu$}:
\begin{equation*}
    J_\nu(u) := \frac{u^\nu}{2^\nu \+ \Gamma\big( \nu + \tfrac{1}{2} \big) \Gamma\big( \tfrac{1}{2} \big)} \int_{-1}^1 e^{\mathrm{i}ut} \big( 1-t^2 \big)^{\nu-1/2} \, dt,
\end{equation*}
where $\Gamma$ is the gamma function. It behaves roughly like $u^{-1/2} \cos(u-b)$ for some phase shift $b = b(\nu)$. The Bessel functions will only arise in the \textit{Hankel transform of order $\nu$}: for any $f \in L_{\mathrm{loc}}^1([0,\infty))$,
\begin{equation*}
    \mathrm{H}_\nu f(s) := \int_0^\infty f(r) J_\nu(sr) \+ r \, dr,
\end{equation*}
where the integral is understood in the limiting sense as an improper integral. The Hankel transform of a polynomial has an explicit expression that we shall state and use later.

A class of functions that ``play well'' with the Bessel functions are the spherical harmonics on $\bbs^{d-1}$---the restrictions to $\bbs^{d-1}$ of the homogeneous harmonic polynomials on $\R^d$---which will play the role of multidimensional Fourier series. More precisely, a function $Y_\ell \!: \bbs^{d-1} \to \C$ is called a \textit{spherical harmonic of degree $\ell$} if there exists a homogeneous polynomial $P \!: \R^d \to \C$ of degree $\ell \in \N$ satisfying $\Delta P = 0$ such that $P|_{\bbs^{d-1}} = Y_\ell$, where $\Delta$ is the Laplacian on $\R^d$. We shall fix an orthonormal basis $\bigcup_{\ell=0}^\infty \{ Y_{\ell,k} \}_{k=1}^{d(\ell)}$ for $L^2(\bbs^{d-1})$, where $d(\ell)$ is the (finite) dimension of the subspace of spherical harmonic of degree $\ell$.

\vspace{-0.25cm} \subsection{Generalized projections} \label{ss:gen-proj} We now turn to the requisite terminology from the Peres--Schlag framework for generalize projections. Let $N,m \in \Z_+$, let $\varnothing \neq U \subseteq \R^N$ be an open set, let $(\Omega,d)$ be a compact metric space, and let $\Pi \!: \Omega \times U \to \R^m$ be a continuous map. Writing $\Pi_{\bm{\lambda}} := \Pi(\,\cdot\,,\bm{\lambda})$, we consider $\Pi$ as a parametrized family $\big( \Pi_{\bm{\lambda}} \!: \Omega \to \R^m \big)_{\bm{\lambda} \in U}$ of mappings from $\Omega$ to $\R^m$.

\begin{defn} \label{defn:gen-proj}
    The family $(\Pi_{\bm{\lambda}})_{\bm{\lambda} \in U}$ is called a \define{family of generalized projections} if $\bm{\lambda} \mapsto \Pi_{\bm{\lambda}}(\bm{\omega})$ is $C^\infty$ for all $\bm{\omega} \in \Omega$ and the following condition is satisfied:
    \begin{enumerate}[topsep=0pt, itemsep=2pt]
        \item[{\normalfont\textbf{i.}}] For every compact set $K \subset U$ and every $\bm{\alpha} \in \N^N$, the inequality
        \begin{equation} \label{eq:bounded-derivs}
            \big| \partial_{\bm{\lambda}}^{\bm{\alpha}} \Pi_{\bm{\lambda}}(\bm{\omega}) \big| \lesssim_{K,|\bm{\alpha}|} 1
        \end{equation}
        holds for all $\bm{\lambda} \in K$ and $\bm{\omega} \in \Omega$.
    \end{enumerate}
    Supposing $(\Pi_{\bm{\lambda}})_{\bm{\lambda} \in U}$ is a family of generalized projections, let
    \begin{equation*}
        \Phi_{\bm{\lambda}}(\bm{\omega},\bm{\zeta}) := \frac{\Pi_{\bm{\lambda}}(\bm{\omega}) - \Pi_{\bm{\lambda}}(\bm{\zeta})}{d(\bm{\omega},\bm{\zeta})}
    \end{equation*}
    for $\bm{\omega} \neq \bm{\zeta}$ and $\Phi_{\bm{\lambda}}(\bm{\omega},\bm{\omega}) := \bm{0}$, and let $\beta \in [0,1)$.
    \begin{enumerate}[topsep=0pt, itemsep=2pt]
        \item[{\normalfont\textbf{ii.}}] The family is said to be \define{transversal of degree $\bm{\beta}$} if, for every compact $K \subset U$, every $\bm{\lambda} \in K$, and all $\bm{\omega} \neq \bm{\zeta}$,
        \begin{equation} \label{eq:transvers}
        \begin{aligned}
            & |\Phi_{\bm{\lambda}}(\bm{\omega},\bm{\zeta})| \lesssim_{\beta,K} d(\bm{\omega},\bm{\zeta})^\beta \\
            &\hspace{2cm} \Longrightarrow \quad \det\!\left[ D_{\bm{\lambda}} \Phi_{\bm{\lambda}}(\bm{\omega},\bm{\zeta}) \big( D_{\bm{\lambda}} \Phi_{\bm{\lambda}}(\bm{\omega},\bm{\zeta}) \big)^{\-\mathrm{t}} \right] \gtrsim_{\beta,K} d(\bm{\omega},\bm{\zeta})^{2\beta}.
        \end{aligned}
        \end{equation}
        \item[{\normalfont\textbf{iii.}}] The family is said to be \define{regular of degree $\bm{\beta}$} if, for every compact $K \subset U$, every $\bm{\lambda} \in K$, all $\bm{\omega} \neq \bm{\zeta}$, and all $\bm{\alpha} \in \N^N$,
        \begin{equation} \label{eq:reg}
        \begin{aligned}
            & |\Phi_{\bm{\lambda}}(\bm{\omega},\bm{\zeta})| \lesssim_{\beta,K} d(\bm{\omega},\bm{\zeta})^\beta \\
            &\hspace{2cm} \Longrightarrow \quad \big| \partial_{\bm{\lambda}}^{\bm{\alpha}} \Phi_{\bm{\lambda}}(\bm{\omega},\bm{\zeta}) \big| \lesssim_{K,\beta,|\bm{\alpha}|} d(\bm{\omega},\bm{\zeta})^{-\beta|\bm{\alpha}|}.
        \end{aligned}
        \end{equation}
        \item[{\normalfont\textbf{iv.}}] The family is said to be \define{strongly regular of degree $\bm{\beta}$} if, for every compact $K \subset U$, every $\bm{\lambda} \in K$, all $\bm{\omega} \neq \bm{\zeta}$, and all $\bm{\alpha} \in \N^N$,
        \begin{equation} \label{eq:strong-reg}
            \big| \partial_{\bm{\lambda}}^{\bm{\alpha}} \Phi_{\bm{\lambda}}(\bm{\omega},\bm{\zeta}) \big| \lesssim_{K,\beta,|\bm{\alpha}|} d(\bm{\omega},\bm{\zeta})^{-\beta|\bm{\alpha}|}
        \end{equation}
        (regardless of the size of $|\Phi_{\bm{\lambda}}(\bm{\omega},\bm{\zeta})|$).
    \end{enumerate}
\end{defn}

The main results of \cite{peres2000smoothness} concern the Sobolev dimension of projections (i.e., pushforwards) of Radon measures. Having a large projection in the direction $\bm{\lambda}$ is certainly requisite for a measure on $\Omega$ to have many large slices by $\Pi_{\bm{\lambda}}$, so we record (a portion of) one of the main Peres--Schlag theorems that bounds the exceptional sets of projections.

\begin{thm}[Peres--Schlag \cite{peres2000smoothness} Theorem 7.3] \label{thm:peres-schlag}
    There exists a constant $a_0 > 0$ depending only on $N$ and $m$ such that the following holds. Let $\big( \Pi_{\bm{\lambda}} \!: \Omega \to \R^m \big)_{\bm{\lambda} \in U}$ be a family of generalized projections that is both transversal and regular of some degree $\beta \in [0,1)$. If $\mu \in \mathcal{M}(\Omega)$ has finite $s$-energy for some $s > m$, then
    \begin{equation*}
        \hdim \big\{ \bm{\lambda} \in U \!: \Pi_{\bm{\lambda}\sharp}\mu \not\!\ll \mathcal{H}^r \big\} \leq N + r - \frac{s}{1 + a_0 \beta}
    \end{equation*}
    for all $r \in (0,m]$.
\end{thm}

An important family of objects in connection with the generalized projections, originally introduced in \cite{orponen2014slicing}, will be the mappings $\Psi_{\bm{\lambda}} \!: \Omega \to \R^{m+mN}$ defined by
\begin{equation} \label{eq:psi-lambda}
    \Psi_{\bm{\lambda}}(\bm{\omega}) := \big( \Pi_{\bm{\lambda}}(\bm{\omega}), D_{\bm{\lambda}} \Pi_{\bm{\lambda}}(\bm{\omega}) \big),
\end{equation}
where the total derivative $D_{\bm{\lambda}} \Pi_{\bm{\lambda}}(\bm{\omega})$ is identified with an element of $\R^{mN}$. In \cite{peres2000smoothness}, $\Pi_{\bm{\lambda}}$ is adequate to relocate the entire projection problem into Euclidean space: the transversality and regularity conditions black-box the fact that preimages of the sets $\Pi_{\bm{\lambda}}(A)$ belong to an abstract metric space. In the slicing problem, on the other hand, the objects of study are precisely those sets that collapse under the application of $\Pi_{\bm{\lambda}}$, so finding another ``natural" mapping of $\Omega$ to a Euclidean space is the first step in adapting the Peres--Schlag framework to the study of slices by means of Fourier analysis.

The maps $\Psi_{\bm{\lambda}}$ do the trick because the transversality condition \eqref{eq:transvers} requires that, modulo a loss inflicted by a positive $\beta$, the first partial derivatives of $\Phi_{\bm{\lambda}}$ pick up the information about pairs of points that $\Phi_{\bm{\lambda}}$ itself loses. It follows from \eqref{eq:transvers} that
\begin{equation*}
    |\Phi_{\bm{\lambda}}(\bm{\omega},\bm{\zeta})| \lesssim_{\beta,K} d(\bm{\omega},\bm{\zeta})^\beta \quad \Longrightarrow \quad \max_{1 \leq i \leq N} \big| \partial_{\lambda_i} \Phi_{\bm{\lambda}}(\bm{\omega},\bm{\zeta}) \big|_\infty \gtrsim_{\beta,K} d(\bm{\omega},\bm{\zeta})^\beta,
\end{equation*}
so packaging the partial derivatives of the components of $\Pi_{\bm{\lambda}}$ into $\Psi_{\bm{\lambda}}$ ensures that it always distinguishes between pairs of points in the same fiber of $\Pi_{\bm{\lambda}}$. Due to \eqref{eq:strong-reg}, $\Psi_{\bm{\lambda}}$ is sufficiently regular as not to stretch the fibers beyond recognition.

We can summarize the relevant properties of $\Psi_{\bm{\lambda}}$ as a single bi-H\"{o}lder condition
\begin{equation} \label{eq:bi-holder}
    d(\bm{\omega},\bm{\zeta})^\beta \lesssim_{\beta,K} \frac{|\Psi_{\bm{\lambda}}(\bm{\omega}) - \Psi_{\bm{\lambda}}(\bm{\zeta})|}{d(\bm{\omega},\bm{\zeta})} \lesssim_{\beta,K} d(\bm{\omega},\bm{\zeta})^{-\beta},
\end{equation}
valid for all $\bm{\lambda}$ belonging to a compact set $K \subset U$, where the lower and upper bounds owe themselves to the transversality and regularity of $\Pi_{\bm{\lambda}}$, respectively. Note that the upper H\"{o}lder bound would not follow if we only required that only \eqref{eq:reg} hold, as in \cite{peres2000smoothness}, rather than \eqref{eq:strong-reg}. However, an important step in the proof of Lemma \ref{lem:technical} requires that $|\Psi_{\bm{\lambda}}(\bm{\omega}) - \Psi_{\bm{\lambda}}(\bm{\zeta})|$ be small whenever $d(\bm{\omega},\bm{\zeta})$ is---not just the $|\Pi_{\bm{\lambda}}(\bm{\omega}) - \Pi_{\bm{\lambda}}(\bm{\zeta})|$ component. In particular, if it is possible to relax this condition, doing so would require a more nuanced strategy for ensuring that the derivative bounds in the hypotheses of Lemma \ref{lem:power-series} are satisfied.

\vspace{-0.25cm} \subsection{Notation and conventions} \label{ss:notation} Throughout, $(\Pi_{\bm{\lambda}})_{\bm{\lambda} \in U}$ denotes a family of generalized projections $\Pi_{\bm{\lambda}} \!: \Omega \to \R^m$, where $(\Omega,d)$ is a compact metric space and $\varnothing \neq U \subseteq \R^N$ is an open set. Typically $(\Pi_{\bm{\lambda}})_{\bm{\lambda} \in U}$ will be both transversal and strongly regular of the same order $\beta \geq 0$. Except in \S \ref{s:applications}, the integers $N,m \in \Z_+$ and $n := mN$ will remain fixed throughout.

\section{A few lemmas} \label{s:lemmas}

The purpose of this section is to begin the proof of Theorem \ref{thm:main} by establishing Lemma \ref{lem:slicing-main}. In broad strokes, the logic parallels that of \cite{peres2000smoothness} \S 7 on generalized projections in arbitrary dimension, but we closely follow the adaptations made by Orponen \cite{orponen2014slicing} for generalized slices. Lemma \ref{lem:slice-energy} is essentially the ``Fubini inequality" for sliced measures and does not have an explicit analogue in \cite{peres2000smoothness}, whereas Lemma \ref{lem:power-series} is taken from \cite{peres2000smoothness} Lemma 3.1 essentially unchanged. The last main ingredient, Lemma \ref{lem:technical}, requires proof but plays the role of \cite{peres2000smoothness} Lemma 7.10.

With $N,m \in \Z_+$ and $n := mN$ fixed as in \S \ref{ss:notation}, let $\mathbf{z} = (\mathbf{x},\mathbf{y}) \in \R^m \times \R^n \cong \R^{m+n}$, $\mathbf{z}^{(1)} := \mathbf{x} \in \R^m$, and $\mathbf{z}^{(2)} := \mathbf{y} \in \R^n$. For $\nu \in \mathcal{M}(\R^{m+n})$ and $\mathbf{x} \in \R^m$, we denote by $\nu_\mathbf{x}^{(1)} \in \mathcal{M}(\R^{m+n})$ the slice of $\nu$ by the coordinate map $(\mathbf{x}',\mathbf{y}') \mapsto \mathbf{x}'$ over the point $\mathbf{x}$.

\vspace{-0.25cm} \subsection{Technical lemmas} \label{ss:lemmas} Lemma \ref{lem:slicing-main} hinges on three results. The first is taken from Mattila \cite{mattila2015fourier}, where it is given a tedious but fairly straightforward proof. This will serve as our starting point in the proof of Lemma \ref{lem:slicing-main}.

\begin{lem}[Mattila \cite{mattila2015fourier} Proposition 6.3] \label{lem:slice-energy}
     For every $\nu \in \mathcal{M}(\R^{m+n})$ and every $m < \sigma < m+n$,
     \begin{equation*}
         \int_{\R^m} I_{\sigma-m}\big( \nu_\mathbf{x}^{(1)} \big) \, d\mathbf{x} \lesssim_{N,m,\sigma} \int_{\R^{m+n}} \big| \mathbf{z}^{(2)} \big|^{\sigma-m-n} \big| \medhat{\nu}(\mathbf{z}) \big|^2 \, d\mathbf{z}.
     \end{equation*}
\end{lem}

The second lemma---a result on the divergence of power series---is paraphrased from \cite{peres2000smoothness}, and it is one of the main pillars in their proofs of their generalized projection theorems. The primary difference between their statement and the one below is the parameter $L$, which they take to be $\infty$; however, as they remark at the beginning of \S 3.2, their proof implies this more general, lower-regularity statement.

\begin{lem}[Peres--Schlag \cite{peres2000smoothness} Lemma 3.1] \label{lem:power-series}
    Let $U \subseteq \R^N$ be an open set, $A,R>1$, and $L \in \Z_+ \cup \{\infty\}$, and suppose $(h_i)_{i=1}^\infty$ is a sequence in $C^L(U)$ satisfying
    \begin{equation} \label{eq:deriv-int-bounds}
        \big\| \partial_{\bm{\lambda}}^{\bm{\alpha}} h_i \big\|_{L^\infty(K)} \lesssim_{K,L} A^{i|\bm{\alpha}|} \quad \text{for all } |\bm{\alpha}| \leq L \quad \text{and} \quad \|h_i\|_{L^1(K)} \lesssim_K R^{-i}
    \end{equation}
    for all $i \in \N$ and every compact set $K \subset U$. If $1 \leq r < R$ and $0 < q < N$ are such that $A^q r^{q/L} \leq \tfrac{R}{r} \leq A^N r^{N/L}$, then
    \begin{equation*}
        \hdim \left\{ \bm{\lambda} \in U \!: \sum_{i=1}^\infty r^i |h_i(\bm{\lambda})| = \infty \right\} \leq N - q.
    \end{equation*}
\end{lem}

Verifying the hypotheses of Lemma \ref{lem:power-series} for certain functions $h_j$ to be defined later constitutes the greater part of the proof of Lemma \ref{lem:slicing-main} in \S \ref{ss:main-lemma} below. Both the derivative and the integral estimates in \eqref{eq:deriv-int-bounds} will demand extensive computations, but the integral bound in particular requires a statement that we break off here as a lemma of its own. Let $\eta \in C_c^\infty((0,\infty))$ be a bump function for some bounded interval. In addition, let $\eta^{(1)}(\mathbf{x}) := \eta(|\mathbf{x}|)$ for all $\mathbf{x} \in \R^m$ and $\eta^{(2)}(\mathbf{y}) := \eta(|\mathbf{y}|)$ for all $\mathbf{y} \in \R^n$.

\begin{lem} \label{lem:technical}
    There exists a constant $a_1 \geq 1$ (depending only on $N$ and $m$) such that, for all $\rho \in C_c^\infty(U)$, all $c \geq 0$, and all $\bm{\omega},\bm{\zeta} \in \Omega$,
    \begin{equation} \label{eq:technical}
    \begin{aligned}
        &\left| \+ \int_U \rho(\bm{\lambda}) \+ \widehat{\eta^{(1)}}\big( 2^j \+ d(\bm{\omega},\bm{\zeta}) \Phi_{\bm{\lambda}}(\bm{\omega},\bm{\zeta}) \big) \widehat{\eta^{(2)}}\big( 2^{j-i} \+ d(\bm{\omega},\bm{\zeta}) D_{\bm{\lambda}} \Phi_{\bm{\lambda}}(\bm{\omega},\bm{\zeta}) \big) \+ d\bm{\lambda} \+ \right| \\
        &\hspace{7.5cm} \lesssim_{\rho,c} \, \big( 1 + 2^j d(\bm{\omega},\bm{\zeta})^{1 + a_1 \beta} \big)^{\!-c}.
    \end{aligned}
    \end{equation}
\end{lem}

This is the higher-dimensional analogue of Orponen \cite{orponen2014slicing} Lemma 4.2, which in turn is an adaptation of Peres--Schlag \cite{peres2000smoothness} Lemma 4.6. The higher-dimensional analogue of the latter also appears in \cite{peres2000smoothness} as Lemma 7.10, and the proof of Lemma \ref{lem:technical} is a lengthy but straightforward amalgam of Orponen's and Peres and Schlag's proofs. For this reason, the argument appears in the \hyperlink{s:appendix}{Appendix}.

\subsection{The main lemma} \label{ss:main-lemma} Recall the family of maps $\Psi_{\bm{\lambda}} \!: \Omega \to \R^m$ defined in \eqref{eq:psi-lambda}. For each $\mu \in \mathcal{M}(\Omega)$,
\begin{equation*}
    \mu_{\bm{\lambda},\mathbf{x}} := \big( \Psi_{\bm{\lambda}\sharp}\mu \big)_\mathbf{x}^{(1)}
\end{equation*}
is the slice of the pushforward measure $\Psi_{\bm{\lambda}\sharp}\mu \in \mathcal{M}(\R^{m+n})$ by the fiber of the coordinate map $(\mathbf{x},\mathbf{y}) \mapsto \mathbf{x}$ over the point $\mathbf{x} \in \R^m$.

\begin{lem} \label{lem:slicing-main}
    There exists a constant $a_2 \geq 1$, depending only on $N$ and $m$, such that the following holds. Let $(\Pi_{\bm{\lambda}})_{\bm{\lambda} \in U}$ be a family of generalized projections that is both transversal and strongly regular of order $\beta$. If $\mu \in \mathcal{M}(\Omega)$ is finite and has finite $(s + a_2 \beta^{1/3})$-energy for some $m < s < N + m$, then
    \begin{equation} \label{eq:exceptional-energy}
        \hdim \left\{ \bm{\lambda} \in U \!: \int_{\R^m} I_{s+\beta^{1/3}-m}(\mu_{\bm{\lambda},\mathbf{x}}) \, d\mathbf{x} = \infty \right\} \leq N + m - s.
    \end{equation}
\end{lem}

\vspace*{-0.2cm}

\begin{proof}
With the notation and hypotheses of the lemma statement and $n := mN$ as before, let $t \geq s$ be such that $I_t(\mu) < \infty$ and let $\sigma := s + \beta^{1/3}$. By Lemma \ref{lem:slice-energy},
\begin{equation*}
    \int_{\R^m} I_{\sigma-m}(\mu_{\bm{\lambda},\mathbf{x}}) \, d\mathbf{x} \lesssim_{\beta,s} \int_{\R^{m+n}} \big| \mathbf{z}^{(2)} \big|^{\sigma-m-n} \big| \widehat{\Psi_{\bm{\lambda}\sharp}\mu}(\mathbf{z}) \big|^2 \, d\mathbf{z}
\end{equation*}
for all $\bm{\lambda} \in U$. Using a countable exhaustion of $U$ by compact sets, we may (and shall) restrict our attention from $U$ to a compact subset $K \subset U$. For each $i \in \N$, let
\begin{equation*}
    \mathbf{C}_i = \left\{ \mathbf{z} \in \R^{m+n} \!: \frac{1}{2^{i+1}} < \frac{\big|\mathbf{z}^{(2)}\big|}{|\mathbf{z}|} \leq \frac{1}{2^i} \right\}
\end{equation*}
be the set of points in $\R^{m+n}$ whose position vectors make an angle between $\cos^{-1}\!\big(2^{-i-1}\big)$ and $\cos^{-1}\!\big(2^{-i}\big)$ with the coordinate plane $\{0\}^m \times \R^n$. These partition $\R^{m+n}$ up to a set of Lebesgue measure $0$, so
\begin{equation} \label{eq:cone-partition}
\begin{aligned}
    & \int_{\R^{m+n}} \big| \mathbf{z}^{(2)} \big|^{\sigma-m-n} \big| \widehat{\Psi_{\bm{\lambda}\sharp}\mu}(\mathbf{z}) \big|^2 \, d\mathbf{z} = \sum_{i=0}^\infty \int_{\mathbf{C}_i} \big| \mathbf{z}^{(2)} \big|^{\sigma-m-n} \big| \widehat{\Psi_{\bm{\lambda}\sharp}\mu}(\mathbf{z}) \big|^2 \, d\mathbf{z} \\
    &\hspace{3cm} \sim_{\beta,s} \sum_{i=0}^\infty 2^{i(n+m-\sigma)} \int_{\mathbf{C}_i} |\mathbf{z}|^{\sigma-m-n} \big| \widehat{\Psi_{\bm{\lambda}\sharp}\mu}(\mathbf{z}) \big|^2 \, d\mathbf{z};
\end{aligned}
\end{equation}
the second line uses the fact that $\big|\mathbf{z}^{(2)}\big| \sim 2^{-i} |\mathbf{z}|$ on $\mathbf{C}_i$. Our goal is to determine how large $t$ must be if this series is to converge for all but a few $\bm{\lambda} \in K$; later, $t$ will take the form $s + a_2 \beta^{1/3}$ for some $a_2 \geq 1$ to be determined.

\textsc{Step 1.} We bound the second series in \eqref{eq:cone-partition} from above by a series of functions amenable to the application of Lemma \ref{lem:power-series}. But first, the $i = 0$ term warrants a separate and elementary treatment: the Plancherel identity \eqref{eq:plancherel-identity} and the transversality constraint \eqref{eq:bi-holder} together yield
\begin{align*}
    \int_{\mathbf{C}_0} |\mathbf{z}|^{\sigma-m-n} \big| \widehat{\Psi_{\bm{\lambda}\sharp}\mu}(\mathbf{z}) \big|^2 \, d\mathbf{z} &\leq \int_{\R^{m+n}} |\mathbf{z}|^{\sigma-m-n} \big| \widehat{\Psi_{\bm{\lambda}\sharp}\mu}(\mathbf{z}) \big|^2 \, d\mathbf{z} \\
    &\sim_{N,m,\beta,s} \iint_{\Omega \times \Omega} \frac{d(\mu \times \mu)(\bm{\omega},\bm{\zeta})}{\big| \Psi_{\bm{\lambda}}(\bm{\omega}) - \Psi_{\bm{\lambda}}(\bm{\zeta}) \big|^\sigma} \\
    &\lesssim_{\beta,K} \iint_{\Omega \times \Omega} \frac{d(\mu \times \mu)(\bm{\omega},\bm{\zeta})}{d(\bm{\omega},\bm{\zeta})^{(1+\beta)\sigma}} \\
    &= I_{(1+\beta)\sigma}(\mu).
\end{align*}
As $\mu$ has finite $t$-energy, this last integral is finite for all $\bm{\lambda} \in K$ provided that $\sigma \leq \tfrac{t}{1+\beta}$.

\vspace*{-0.05cm}

To bound the remaining terms in \eqref{eq:cone-partition}, we increase the regularity of the integrals using smooth bump functions $\phi_i \!: \R^{m+n} \to [0,1]$ satisfying
\begin{equation*}
    \chi_{\mathbf{C}_i} \leq \phi_i \leq \chi_{\mathbf{C}_{i-1} \cup \mathbf{C}_i \cup \mathbf{C}_{i+1}}.
\end{equation*}

\newpage

In particular, we set
\begin{equation*}
    \phi_i(\mathbf{z}) := \phi\big(2^{i+1}\big|\mathbf{z}^{(2)}\big|/|\mathbf{z}|\big), \quad \text{where} \quad \phi \in C_c^\infty\big( \big( \tfrac{1}{2}, \tfrac{5}{2} \big) \big)
\end{equation*}
is a smooth bump function for the interval $[1,2]$; thus,
\begin{equation*}
    \int_{\mathbf{C}_i} |\mathbf{z}|^{\sigma-m-n} \big| \widehat{\Psi_{\bm{\lambda}\sharp}\mu}(\mathbf{z}) \big|^2 \, d\mathbf{z} \leq \int_{\R^{m+n}} \phi_i(\mathbf{z}) |\mathbf{z}|^{\sigma-m-n} \big| \widehat{\Psi_{\bm{\lambda}\sharp}\mu}(\mathbf{z}) \big|^2 \, d\mathbf{z}
\end{equation*}
for all $i \geq 1$. In view of our aim of proving that the value in \eqref{eq:cone-partition} is finite for ``most" $\bm{\lambda} \in K$, our work is complete if we can show that
\begin{equation*}
    \hdim \left\{ \bm{\lambda} \in K \!: \sum_{i=1}^\infty 2^{i(2n+m-\sigma)} h_i(\bm{\lambda}) = \infty \right\} \leq N + m - s
\end{equation*}
for all sufficiently small $\beta \geq 0$, where
\begin{equation} \label{eq:h-i}
    h_i(\bm{\lambda}) := \frac{1}{2^{in}} \int_{\R^{m+n}} \phi_i(\mathbf{z}) |\mathbf{z}|^{\sigma-m-n} \big| \widehat{\Psi_{\bm{\lambda}\sharp}\mu}(\mathbf{z}) \big|^2 \, d\mathbf{z}.
\end{equation}
We do so by means of Lemma \ref{lem:power-series}, and the remainder of this proof consists primarily of showing that, for some $0 < p < n$, these $h_i$ satisfy \eqref{eq:deriv-int-bounds} with $A = 2$, $R = 2^{p+n}$, $q < s-m$, and some large $L \in \Z_+$ to be determined. As such, we must bound the $L^1$-norms of the $h_i|_K$ and the $L^\infty$-norms of their derivatives.

\textsc{Step 2.} We verify the $L^1$-bounds in \eqref{eq:deriv-int-bounds}. Let $A_{i,j} := \big\{ \mathbf{z} \in \spt \phi_i \!: 2^j \leq |\mathbf{z}| < 2^{j+1} \big\}$ for $j \in \Z$, so that
\begin{equation} \label{eq:dyadic-h-i}
    h_i(\bm{\lambda}) \sim \frac{1}{2^{in}} \sum_{j \in \Z} 2^{j(\sigma-m-n)} \int_{A_{i,j}} \big| \widehat{\Psi_{\bm{\lambda}\sharp}\mu}(\mathbf{z}) \big|^2 \, d\mathbf{z}
\end{equation}
by the definition of $h_i$ and the boundedness of the Fourier transform of a finite Borel measure. For a sufficiently large constant $b \geq 2$ depending only on $N$ and $m$, the set $A_{i,j}$ is contained in the product annulus
\begin{equation*}
    B_{i,j} := \left\{ \mathbf{z} \in \R^{m+n} \!: \tfrac{1}{b} 2^j < \big| \mathbf{z}^{(1)} \big| \leq b 2^j, \ \tfrac{1}{b} 2^{j-i} < \big| \mathbf{z}^{(2)} \big| \leq b 2^{j-i} \right\}.
\end{equation*}
If $\eta \!: \big( \tfrac{1}{b^2}, b^2 \big) \to [0,1]$ is a smooth bump function for the interval $\big[ \tfrac{1}{b}, b \big]$ and if $\eta_j^{(k)}\-(\mathbf{z}) := \eta\big( 2^{-j} \big| \mathbf{z}^{(k)} \big| \big)$ for $\mathbf{z} \in \R^{m+n}$, then $\mathbf{z} \mapsto \eta_j^{(1)}\-(\mathbf{z}) \eta_{j-i}^{(2)}\-(\mathbf{z})$ is a smooth bump function for $B_{i,j}$. It then follows from the Plancherel identity \eqref{eq:plancherel-identity}, the definition of $\Psi_{\bm{\lambda}}$, and the tensor product factorization of the Fourier transform that
\begin{equation} \label{eq:phi-on-cylinder}
\begin{aligned}
    & \int_{A_{i,j}} \big| \widehat{\Psi_{\bm{\lambda}\sharp}\mu}(\mathbf{z}) \big|^2 \, d\mathbf{z} \leq \int_{\R^{m+n}} \eta_j^{(1)}\-(\mathbf{z}) \eta_{j-i}^{(2)}(\mathbf{z}) \big| \widehat{\Psi_{\bm{\lambda}\sharp}\mu}(\mathbf{z}) \big|^2 \, d\mathbf{z} \\
    &\hspace{20pt} = \iint_{\Omega \times \Omega} \big( \eta_j^{(1)} \eta_{j-i}^{(2)} \big)\!\widehat{\raisebox{6pt}{ \ }}\big( \Psi_{\bm{\lambda}}(\bm{\omega}) - \Psi_{\bm{\lambda}}(\bm{\zeta}) \big) \+ d(\mu \times \mu)(\bm{\omega},\bm{\zeta}) \\
    &\hspace{20pt} = 2^{mj+n(j-i)} \iint_{\Omega \times \Omega} \widehat{\eta^{(1)}}\big( 2^j d(\bm{\omega},\bm{\zeta}) \Phi_{\bm{\lambda}}(\bm{\omega},\bm{\zeta}) \big) \\
    &\hspace{4.5cm} \cdot \widehat{\eta^{(2)}}\big( 2^{j-i} d(\bm{\omega},\bm{\zeta}) D_{\bm{\lambda}} \Phi_{\bm{\lambda}}(\bm{\omega},\bm{\zeta}) \big) \+ d(\mu \times \mu)(\bm{\omega},\bm{\zeta}).
\end{aligned}
\end{equation}
In the third line, $\eta^{(1)}(\mathbf{x}) := \eta_0^{(1)}(\mathbf{x},\mathbf{y})$ for any $\mathbf{y} \in \R^n$ and $\eta^{(2)}(\mathbf{y}) := \eta_0^{(2)}(\mathbf{x},\mathbf{y})$ for any $\mathbf{x} \in \R^m$ (note that $\eta_0^{(1)}$ is independent of $\mathbf{y}$ and $\eta_0^{(2)}$ is independent of $\mathbf{x}$), and the Fourier transforms thereof denote the $m$- and $n$-dimensional Fourier transforms, respectively, rather than the full $(m+n)$-dimensional Fourier transform.

Fix a value $0 < p < n$, to be chosen sufficiently close to $n$ much later in the proof. Lemma \ref{lem:technical}, \eqref{eq:dyadic-h-i}, and \eqref{eq:phi-on-cylinder} together imply that, whenever $\rho \in C_c^\infty(U)$ is a smooth bump function for $K$,
\begin{equation} \label{eq:int-h-i-estimate}
\begin{aligned}
    &2^{i(p+n)} \int_K h_i(\bm{\lambda}) \, d\bm{\lambda} \lesssim_\rho 2^{ip} \sum_{j \in \Z} 2^{j(\sigma-m-n)} \int_U \rho(\bm{\lambda}) \int_{A_{i,j}} \big| \widehat{\Psi_{\bm{\lambda}\sharp}\mu}(\mathbf{z}) \big|^2 \, d\mathbf{z} \+ d\bm{\lambda} \\
    &\qquad \lesssim_K 2^{ip} \sum_{j \in \Z} 2^{j(\sigma-m-n)} \cdot 2^{mj+n(j-i)} \\
    &\hspace{3cm} \cdot \iint_{\Omega \times \Omega} \bigg[ \int_U \rho(\bm{\lambda}) \widehat{\eta^{(1)}}\big( 2^j d(\bm{\omega},\bm{\zeta}) \Phi_{\bm{\lambda}}(\bm{\omega},\bm{\zeta}) \big) \\
    &\hspace{4cm} \cdot \widehat{\eta^{(2)}}\big( 2^{j-i} d(\bm{\omega},\bm{\zeta}) D_{\bm{\lambda}} \Phi_{\bm{\lambda}}(\bm{\omega},\bm{\zeta}) \big) d\bm{\lambda} \bigg] \+ d(\mu \times \mu)(\bm{\omega},\bm{\zeta}) \\
    &\qquad \lesssim 2^{i(p-n)} \sum_{j \in \Z} 2^{j\sigma} \iint_{\Omega \times \Omega} \big( 1 + 2^j d(\bm{\omega},\bm{\zeta})^{1 + a_1 \beta} \big)^{-c} \, d(\mu \times \mu)(\bm{\omega},\bm{\zeta}).
\end{aligned}
\end{equation}
(The implicit constant depends only on $K$ and not $\rho$ once a standard procedure for constructing bump functions chosen.) The series in $j$ converges provided $c > \sigma$, and in fact it is bounded above by $\sim_{\beta,s,c} I_{(1 + a_1 \beta)\sigma}(\mu)$. To see this, write
\begin{align*}
    \sum_{j \in \Z} \frac{2^{js}}{\big( 1 + 2^j d(\bm{\omega},\bm{\zeta})^{1 + a_1 \beta} \big)^c} &\sim d(\bm{\omega},\bm{\zeta})^{-(1 + a_1 \beta)\sigma} \left( \sum_{k=0}^\infty \frac{2^{ks}}{{(1 + 2^k)}^c} + \sum_{k = -\infty}^{-1} \frac{2^{ks}}{{(1 + 2^k)}^c} \right) \\
    &\leq d(\bm{\omega},\bm{\zeta})^{-(1 + a_1 \beta)\sigma} \left( \frac{1}{1 - 2^{\sigma-c}} + \frac{1}{1 - 2^{-\sigma}} \right),
\end{align*}
so that by \eqref{eq:int-h-i-estimate}
\begin{equation*}
    2^{i(p+n)} \int_K h_i(\bm{\lambda}) \, d\bm{\lambda} \lesssim_{\beta,K,s,c} 2^{i(p-n)} \iint_{\Omega \times \Omega} \frac{d(\mu \times \mu)(\bm{\omega},\bm{\zeta})}{d(\bm{\omega},\bm{\zeta})^{(1 + a_1 \beta)\sigma}} = 2^{i(p-n)} I_{(1 + a_1 \beta)\sigma}(\mu).
\end{equation*}
With (say) $c = N+m > \sigma$, the power series estimate
\begin{equation*}
    \sum_{i=1}^\infty 2^{i(p+n)} \| h_i \|_{L^1(K)} \lesssim_{N,m,\beta,K,s} \frac{1}{1 - 2^{p-n}} I_{(1 + a_1 \beta)\sigma}(\mu)
\end{equation*}
follows at once, so
\begin{equation*}
    \| h_i \|_{L^1(K)} \lesssim_{N,m,\beta,K,s,p} 2^{-i(p+n)}.
\end{equation*}
Note that this imposes the new constraint $\sigma \leq \tfrac{t}{1 + a_1 \beta}$, which is stronger than previously-established the constraint $\sigma \leq \tfrac{t}{1+\beta}$ since $a_1 \geq 1$.

\textsc{Step 3.} We verify the $L^\infty$-bounds \eqref{eq:deriv-int-bounds} on the partial derivatives $\partial_{\bm{\lambda}}^{\bm{\alpha}} h_i$. To do so, we rewrite the integral in \eqref{eq:h-i} using Plancherel's theorem and repeatedly differentiate under the integral sign. Plancherel's theorem tells us that
\begin{equation} \label{eq:h-i-omega}
    h_i(\bm{\lambda}) = \frac{1}{2^{in}} \iint_{\Omega \times \Omega} \widehat{g_i}\big( \Psi_{\bm{\lambda}}(\bm{\omega}) - \Psi_{\bm{\lambda}}(\bm{\zeta}) \big) \, d(\mu \times \mu)(\bm{\omega},\bm{\zeta}),
\end{equation}
where
\begin{equation*}
    g_i(\mathbf{z}) := |\mathbf{z}|^{\sigma-m-n} \phi_i(\mathbf{z}) = |\mathbf{z}|^{\sigma-m-n} \phi\big(2^{i+1}\big|\mathbf{z}^{(2)}\big|/|\mathbf{z}|\big).
\end{equation*}
Our goal, then, is to establish pointwise bounds on the partial derivatives of the function
\begin{equation} \label{eq:gamma-i}
    \gamma_i(\bm{\lambda}) := \widehat{g_i}\big( \Psi_{\bm{\lambda}}(\bm{\omega}) - \Psi_{\bm{\lambda}}(\bm{\zeta}) \big),
\end{equation}
where $\bm{\omega},\bm{\zeta} \in \Omega$ are fixed. To do this, we first consider $\widehat{g_i}$ itself. Write $\mathbf{z} = r\bm{\nu}$ in spherical coordinates, so that
\begin{equation} \label{eq:spherical-g-i-hat}
    \widehat{g_i}(\mathbf{w}) = \int_0^\infty r^{\sigma-1} \int_{\bbs^{m+n-1}} \phi_i(\bm{\nu}) e^{-2\pi\mathrm{i}r \bm{\nu} \cdot \mathbf{w}} \, d\mathcal{H}^{m+n-1}(\bm{\nu}) \+ dr.
\end{equation}
Since $\phi_i |_{\bbs^{m+n-1}}$ is $C^\infty$ (it is a smooth bump function for a ``band'' on the sphere of width $\sim 2^{-i}$), it has an absolutely convergent spherical harmonic expansion
\begin{equation*}
    \phi_i(\bm{\nu}) = \sum_{\ell=0}^\infty \sum_{k=1}^{d(\ell)} a_{i,\ell,k} Y_{\ell,k}(\bm{\nu})
\end{equation*}
with rapidly-decaying coefficients. In fact, since $\phi_i(\bm{\nu}) = \phi\big( 2^{i+1} \big| \bm{\nu}^{(2)} \big| \big)$, there exist $a_{\ell,k}$ such that $|a_{i,\ell,k}| \sim 2^{-in} |a_{k,\ell}|$ for all $i \in \Z_+$---a fact we shall use shortly.

Now, the identity
\begin{align*}
    & \int_{\bbs^{m+n-1}} Y_{\ell,k}(\bm{\nu}) e^{-2\pi \mathrm{i}r \bm{\nu} \cdot \mathbf{w}} \, d\mathcal{H}^{m+n-1}(\bm{\nu}) \\
    &\hspace{2.5cm} = (2\pi)^{(m+n-1)/2} (-i)^\ell \frac{J_{\ell+(m+n-2)/2}(2\pi r|\mathbf{w}|)}{(r|\mathbf{w}|)^{(m+n-2)/2}} Y_{\ell,k}\!\left( \frac{\mathbf{w}}{|\mathbf{w}|} \right)
\end{align*}
is standard, from which it follows that
\begin{align*}
    & \int_{\bbs^{m+n-1}} \phi_i(\bm{\nu}) e^{-2\pi\mathrm{i}r \bm{\nu} \cdot \mathbf{w}} \, d\mathcal{H}^{m+n-1}(\bm{\nu}) \\
    &\hspace{2.5cm} = (2\pi)^{(m+n-1)/2} \sum_{\ell=0}^\infty \sum_{k=1}^{d(\ell)} (-i)^\ell a_{i,\ell,k} \frac{J_{\ell+(m+n-2)/2}(2\pi r|\mathbf{w}|)}{(r|\mathbf{w}|)^{(m+n-2)/2}} Y_{\ell,k}\!\left( \frac{\mathbf{w}}{|\mathbf{w}|} \right).
\end{align*}
Plugging this back into \eqref{eq:spherical-g-i-hat} then yields
\begin{align*}
    & \widehat{g_i}(\mathbf{w}) \simeq_{m,N} \sum_{\ell=0}^\infty \sum_{k=1}^{d(\ell)} (-i)^\ell a_{i,\ell,k} \+ Y_{\ell,k}\!\left( \frac{\mathbf{w}}{|\mathbf{w}|} \right) \! \int_0^\infty r^{\sigma-1} \frac{J_{\ell+(m+n-2)/2}(2\pi r|\mathbf{w}|)}{(r|\mathbf{w}|)^{(m+n-2)/2}} \, dr \\
    &\quad = |\mathbf{w}|^{-(m+n-2)/2} \sum_{\ell=0}^\infty \sum_{k=1}^{d(\ell)} (-i)^\ell a_{i,\ell,k} \+ Y_{\ell,k}\!\left( \frac{\mathbf{w}}{|\mathbf{w}|} \right) \! \int_0^\infty r^{\sigma-(m+n)/2} J_{\ell+(m+n-2)/2}(2\pi r|\mathbf{w}|) \, dr.
\end{align*}
The integral in the second line is precisely the Hankel transform of order $\ell+(m+n-2)/2$ of the monomial $r^{\sigma-(m+n-2)/2}$, and it can be expressed simply in terms of the gamma function as
\begin{equation*}
    \frac{2^{\sigma-(m+n)/2}}{(2\pi|\mathbf{w}|)^{\sigma-(m+n+2)/2}} \frac{\Gamma\!\left( \frac{\ell+\sigma+1}{2} \right)}{\Gamma\!\left( \frac{\ell-\sigma+m+n}{2} \right)}.
\end{equation*}
Thus, we have shown that
\begin{equation*}
    \widehat{g_i}(\mathbf{w}) \simeq_{m,N,s,\beta} |\mathbf{w}|^{-\sigma} \sum_{\ell=0}^\infty \sum_{k=1}^{d(\ell)} (-i)^\ell a_{i,\ell,k} \+ Y_{\ell,k}\!\left( \frac{\mathbf{w}}{|\mathbf{w}|} \right) \frac{\Gamma\!\left( \frac{\ell+\sigma+1}{2} \right)}{\Gamma\!\left( \frac{\ell-\sigma+m+n}{2} \right)}.
\end{equation*}
Now consider $\partial^{\bm{\alpha}} \widehat{g_i}(\mathbf{w})$ for $\bm{\alpha} = \big( \bm{\alpha}^{(1)}, \bm{\alpha}^{(2)} \big) \in \N^m \times \N^n$. We first have the elementary estimate
\begin{equation*}
    \partial^{\bm{\alpha}} |\mathbf{w}|^{-\sigma} \lesssim_{|\bm{\alpha}|} \sigma^{|\bm{\alpha}|} |\mathbf{w}|^{-\sigma-|\bm{\alpha}|} \sim_{\beta,s,|\bm{\alpha}|} |\mathbf{w}|^{-\sigma-|\bm{\alpha}|}.
\end{equation*}
Denoting by $P_{\ell,k}$ an extension of $Y_{\ell,k}$ to a harmonic polynomial of degree $\ell$ on $\R^{m+n}$, so that $Y_{\ell,k}(\mathbf{w}/|\mathbf{w}|) = |\mathbf{w}|^{-\ell} P_{\ell,k}(\mathbf{w})$, we can similarly estimate the derivatives of the $Y_{\ell,k}$ factor:
\begin{align*}
    \bigg| \partial^{\bm{\alpha}} Y_{\ell,k}\!\left( \frac{\mathbf{w}}{|\mathbf{w}|} \right) \bigg| &\leq \sum_{|\bm{\gamma}| \leq |\bm{\alpha}|} \binom{|\bm{\alpha}|}{|\bm{\gamma}|} \big| \partial^{\bm{\gamma}} |\mathbf{w}|^{-\ell} \big| \big| \partial^{\bm{\alpha}-\bm{\gamma}} P_{\ell,k}(\mathbf{w}) \big| \\
    &\lesssim_{|\bm{\alpha}|} |\mathbf{w}|^{-\ell-|\bm{\alpha}|} \cdot \ell^{|\bm{\alpha}|} |\mathbf{w}|^{\ell} \| Y_{\ell,k} \|_{L^\infty(\bbs^{m+n-1})} \\
    &\lesssim \ell^{|\bm{\alpha}|+(m+n-2)/2} |\mathbf{w}|^{-|\bm{\alpha}|}.
\end{align*}
The bound $\| Y_{\ell,k} \|_{L^\infty(\bbs^{m+n-1})} \lesssim \ell^{(m+n-2)/2}$ arises from the Sturm--Liouville construction of the basis of spherical harmonics from the Laplacian on $\mathbb{S}^{m+n-1}$. Together, the above inequalities, the asymptotic formula
\begin{equation*}
    \frac{\Gamma(x+r)}{\Gamma(x)} \sim x^r
\end{equation*}
for all real $x,r \geq 0$, and the Leibniz rule give
\begin{align*}
    & \big| \partial^{\bm{\alpha}} \widehat{g_i}(\mathbf{w}) \big| \\
    &\quad \lesssim_{m,N,\beta,s} |\mathbf{w}|^{-\sigma-|\bm{\alpha}|} \sum_{\ell=0}^\infty \sum_{k=1}^{d(\ell)} |a_{i,\ell,k}| \ell^{|\bm{\alpha}|+(m+n-2)/2} |\mathbf{w}|^{-\ell} \- \left( \frac{\ell-\sigma+m+n}{2} \right)^{\!\sigma-(m+n-1)/2} \\
    &\quad \lesssim_{m,N} |\mathbf{w}|^{-\sigma-|\bm{\alpha}|} \sum_{\ell=0}^\infty \sum_{k=1}^{d(\ell)} |a_{i,\ell,k}| \+ \ell^{|\bm{\alpha}|+\sigma-1/2} |\mathbf{w}|^{-\ell} \\
    &\quad \lesssim |\mathbf{w}|^{-\sigma-|\bm{\alpha}|} \sum_{\ell=0}^\infty \sum_{k=1}^{d(\ell)} 2^{-in} |a_{\ell,k}| \+ \ell^{|\bm{\alpha}|+\sigma-1/2} |\mathbf{w}|^{-\ell} \\
    &\quad \lesssim 2^{i(|\bm{\alpha}|-n)} |\mathbf{w}|^{-\sigma-|\bm{\alpha}|},
\end{align*}
where in the final line we have used the fact that the coefficients $|a_{\ell,k}|$ are rapidly-decaying and, consequently, are $\lesssim_{\bm{\alpha}} \tfrac{1}{d(\ell)} \ell^{-|\bm{\alpha}|-\sigma-1/2} \sim_{m,N} \ell^{-|\bm{\alpha}|-\sigma-m-n+1/2}$. The factor $|\mathbf{w}|^{-\ell}$ can be eliminated independently of $\mathbf{w}$ bounded away from $\mathbf{0}$. In particular, evaluating at $\mathbf{w} = \Psi_{\bm{\lambda}}(\bm{\omega}) - \Psi_{\bm{\lambda}}(\bm{\zeta})$ gives, by \eqref{eq:bi-holder},
\begin{align*}
    \big| \partial_\mathbf{w}^{\bm{\alpha}} \widehat{g_i}\big(\Psi_{\bm{\lambda}}(\bm{\omega}) - \Psi_{\bm{\lambda}}(\bm{\zeta})\big) \big| &\lesssim_{\beta,s,|\bm{\alpha}|} 2^{i(|\bm{\alpha}|-n)} \big( 1 + d(\bm{\omega},\bm{\zeta})^{\beta+1} \big)^{-\sigma-|\bm{\alpha}|} \\
    &\leq 2^{i(|\bm{\alpha}|-n)} d(\bm{\omega},\bm{\zeta})^{-(1+\beta)(\sigma+|\bm{\alpha}|)}.
\end{align*}
Recall that our intention was to establish bounds on the partial derivatives of $\gamma_i$. It is clear that, for every multi-index $\bm{\alpha} \in \N^N$, the corresponding partial derivative $\partial_{\bm{\lambda}}^{\bm{\alpha}} \gamma_i(\bm{\lambda})$ is a finite sum of terms the form
\begin{equation*}
    \partial_\mathbf{w}^{\bm{\alpha_0}} \widehat{g_i}(\mathbf{w}) \, \partial_{\bm{\lambda}}^{\bm{\alpha_1}} \- \big[ \Psi_{1,{\bm{\lambda}}}(\bm{\omega}) - \Psi_{1,{\bm{\lambda}}}(\bm{\zeta}) \big] \cdots \, \partial_{\bm{\lambda}}^{\bm{\alpha_{m+n}}} \- \big[ \Psi_{m+n,{\bm{\lambda}}}(\bm{\omega}) - \Psi_{m+n,{\bm{\lambda}}}(\bm{\zeta}) \big] \+ \Big|_{\mathbf{w} = \Psi_{\bm{\lambda}}(\bm{\omega}) - \Psi_{\bm{\lambda}}(\bm{\zeta})} \, ,
\end{equation*}
where $\Psi_{j,\bm{\lambda}}$ is the $j$\textsuperscript{th} component of $\Psi_{\bm{\lambda}}$, each $\bm{\alpha_j} \in \N^{m+n}$ has length $|\bm{\alpha_j}| \leq |\bm{\alpha}|$, and (contrary to the usual convention) the $j$\textsuperscript{th} term is simply $1$ when $\bm{\alpha_j} = \mathbf{0}$. (See \cite{orponen2014slicing} Proposition 4.3.)

Therefore, to prove the desired estimates on the partial derivatives of $\gamma_i$, we need only to estimate the contributions of each of the $k+1$ factors. Having treated the first factor already, we turn our attention to $\partial_{\bm{\lambda}}^{\bm{\alpha_j}} \- \big[ \Psi_{j,\bm{\lambda}}(\bm{\omega}) - \Psi_{j,\bm{\lambda}}(\bm{\zeta}) \big]$. These follow readily from the regularity hypotheses on $\Pi_{\bm{\lambda}}$: if $m+1 \leq j \leq n$, then for some $1 \leq k \leq N$
\begin{align*}
    \big| \partial_{\bm{\lambda}}^{\bm{\alpha_j}} \- \big[ \Psi_{j,\bm{\lambda}}(\bm{\omega}) - \Psi_{j,\bm{\lambda}}(\bm{\zeta}) \big] \big| &= \big| \partial_{\bm{\lambda}}^{\bm{\alpha_j}} \big[ d(\bm{\omega},\bm{\zeta}) \+ \partial_{\lambda_k} \- \Phi_{\bm{\lambda}}(\bm{\omega},\bm{\zeta}) \big] \big| \\
    &= d(\bm{\omega},\bm{\zeta}) \big| \partial_{\bm{\lambda}}^{\bm{\alpha_j}} \partial_{\lambda_k} \- \Phi_{\bm{\lambda}}(\bm{\omega},\bm{\zeta}) \big| \\
    &\lesssim_{\beta,s,|\bm{\alpha}|} d(\bm{\omega},\bm{\zeta}) \cdot d(\bm{\omega},\bm{\zeta})^{-\beta(|\bm{\alpha_j}|+1)} \\
    &\lesssim d(\bm{\omega},\bm{\zeta})^{1-\beta(|\bm{\alpha}|+1)},
\end{align*}
and similarly when $1 \leq j \leq m$ except that $d(\bm{\omega},\bm{\zeta})^{1-\beta(|\bm{\alpha}|+1)}$ is replaced by the ``smaller'' (up to a multiplicative constant depending on $|\Omega|$) factor $d(\bm{\omega},\bm{\zeta})^{1-\beta|\bm{\alpha}|}$. There are $k = |\bm{\alpha_0}|$ such factors, which together with the factor $\partial_\mathbf{z}^{\bm{\alpha_0}} \widehat{g_i}(\mathbf{z})$ yields
\begin{align*}
    \big| \partial_{\bm{\lambda}}^{\bm{\alpha}} \gamma_i(\bm{\lambda}) \big| &\lesssim_{\beta,s,|\bm{\alpha}|} 2^{i(|\bm{\alpha_0}|-n)} d(\bm{\omega},\bm{\zeta})^{-(1+\beta)(\sigma+|\bm{\alpha_0}|)} \cdot d(\bm{\omega},\bm{\zeta})^{|\bm{\alpha_0}|(1-\beta(|\bm{\alpha}|+1))} \\
    &\lesssim_{|\bm{\alpha}|} 2^{i(|\bm{\alpha}|-n)} d(\bm{\omega},\bm{\zeta})^{-\sigma - \beta(|\bm{\alpha}|^2 + 2|\bm{\alpha}| + \sigma)}.
\end{align*}
On the other hand, by the basic uniform estimates on $\partial_{\bm{\lambda}}^{\bm{\alpha}} \Pi_{\bm{\lambda}}$ in \eqref{eq:bounded-derivs},
\begin{equation*}
    \big| \partial_{\bm{\lambda}}^{\bm{\alpha_j}} \- \big[ \Psi_{j,\bm{\lambda}}(\bm{\omega}) - \Psi_{j,\bm{\lambda}}(\bm{\zeta}) \big] \big| \leq \big| \partial_{\bm{\lambda}}^{\bm{\alpha}} \partial_{\lambda_k} \Pi_{\bm{\lambda}}(\bm{\omega}) \big| + \big| \partial_{\bm{\lambda}}^{\bm{\alpha}} \partial_{\lambda_k} \Pi_{\bm{\lambda}}(\bm{\zeta}) \big| \lesssim_{K,|\bm{\alpha}|,\beta} 1.
\end{equation*}
This gives rise to the bound
\begin{align*}
    \big| \partial_{\bm{\lambda}}^{\bm{\alpha}} \gamma_i(\bm{\lambda}) \big| &\lesssim_{\beta,K,|\bm{\alpha}|} 2^{i(|\bm{\alpha}|-n)} d(\bm{\omega},\bm{\zeta})^{-(1+\beta)(\sigma+|\bm{\alpha}|)},
\end{align*}
which improves upon the previous when $\beta > 0$ and $|\bm{\alpha}|+1 > \beta^{-1}$. Thus, with $\tau = \tau_{\beta,s,|\bm{\alpha}|} := \min \big\{ \sigma + \beta(|\bm{\alpha}|^2 + 2|\bm{\alpha}| + \sigma),  (1+\beta)(\sigma+|\bm{\alpha}|) \big\}$,
\begin{equation*}
    \big| \partial_{\bm{\lambda}}^{\bm{\alpha}} \gamma_i(\bm{\lambda}) \big| \lesssim_{\beta,K,s,|\bm{\alpha}|} 2^{i(|\bm{\alpha}|-n)} d(\bm{\omega},\bm{\zeta})^{-\tau}.
\end{equation*}
Recalling the definition of $\gamma_i$ in \eqref{eq:gamma-i}, we may now differentiate \eqref{eq:h-i-omega} under the integral sign and use this inequality to conclude that
\begin{equation*}
    \big| \partial_{\bm{\lambda}}^{\bm{\alpha}} h_i(\bm{\lambda}) \big| \lesssim \frac{2^{i(|\bm{\alpha}|-n)}}{2^{in}} \iint_{\Omega \times \Omega} \frac{d(\mu \times \mu)(\bm{\omega},\bm{\zeta})}{d(\bm{\omega},\bm{\zeta})^\tau} = 2^{i(|\bm{\alpha}|-2n)} I_\tau(\mu) \leq 2^{i|\bm{\alpha}|} I_\tau(\mu)
\end{equation*}
for all $\bm{\lambda} \in K$. The right-hand side is finite when $\tau \leq t$, which fact we shall use momentarily in deciding for which $L \in \Z_+$ Lemma \ref{lem:power-series} applies.

\textsc{Step 4.} Having established the bounds in \eqref{eq:deriv-int-bounds} on the functions $h_i$ defined in \eqref{eq:h-i}, we at last apply Lemma \ref{lem:power-series} to conclude the proof. The case $\beta = 0$ does not quite follow from the case $\beta > 0$, and since the latter is slightly more involved, we begin with the former. Let
\begin{equation*}
    0 < q <  s - m, \hspace{0.25cm} n - (s - m - q) < p < n, \hspace{0.25cm} A := 2, \hspace{0.25cm} R := 2^{p+n}, \hspace{0.25cm} \text{and} \hspace{0.25cm} r := 2^{2n+m-\sigma},
\end{equation*}
and let $L \in \Z_+$ be sufficiently large that
\begin{equation*}
    q + \left( \frac{q}{L} + 1 \right) (2n+m-\sigma) \leq p.
\end{equation*}
We must verify that $A^q r^{q/L} \leq \tfrac{R}{r} \leq A^N r^{N/L}$. Since all three expressions are powers of $2$, this is equivalent to
\begin{equation*}
    q + \left( \frac{q}{L} + 1 \right)(2n+m-\sigma) \leq p+n \leq N + \left( \frac{N}{L} + 1 \right)(2n+m-\sigma),
\end{equation*}
which is indeed consistent with the choices of parameters made above owing to the fact that $\sigma = s$:
\begin{align*}
    q + \left( \frac{q}{L} + 1 \right)(2n+m-\sigma) &= (n - (s-m-q)) + n + \frac{q}{L} (2n+m-s) \\
    &\leq p+n = N + (p+n-N) \leq N + (2n - (s-m)) \\
    &\leq N + \left( \frac{q}{L} + 1 \right)(2n+m-\sigma).
\end{align*}
The constraints on $t$ accumulated over the course of the proof are
\begin{equation} \label{eq:beta-constraints}
    \sigma \leq \frac{t}{1 + a_1 \beta} \quad \text{and} \quad \min \big\{ \sigma + \beta(L^2 + 2L + t), (1+\beta)(\sigma+L) \big\} \leq t,
\end{equation}
both of which hold trivially with $t = s$ when $\beta = 0$. Thus, our parameter choices satisfy the hypotheses of Lemma \ref{lem:power-series}, so by the results of Step 1, \eqref{eq:exceptional-energy} holds.

Now suppose instead that $0 < \beta < 1$ and let
\begin{equation*}
    q := s - m, \quad p := n - \frac{1}{2} \beta^{1/3}, \quad A := 2, \quad R := 2^{p+n}, \quad \text{and} \quad r := 2^{2n+m-\sigma}.
\end{equation*}
If $L \in \big[ \frac{4nN}{\beta^{1/3}}, \+ \frac{8nN}{\beta^{1/3}} \big] \cap \Z$, then the hypotheses of Lemma \ref{lem:power-series} are satisfied just as in the $\beta = 0$ case:
\begin{align*}
    q + \left( \frac{q}{L} + 1 \right)(2n + m - \sigma) &< \left( n - \frac{1}{2} \beta^{1/3} \right) + n - \frac{1}{2} \beta^{1/3} + \frac{s-m}{L} (2n + m - s - \beta^{1/3}) \\
    &< \left( n - \frac{1}{2} \beta^{1/3} \right) + n - \frac{1}{2} \beta^{1/3} + \frac{2nN}{L} \\
    &\leq p + n \\
    &= N + \left( 2n - N - \frac{1}{2} \beta^{1/3} \right) \\
    &\leq N + \left( \frac{N}{L} + 1 \right) (2n + m - \sigma).
\end{align*}
In addition, the first inequality in \eqref{eq:beta-constraints} is satisfied by any $t \geq s + b_0 \beta^{1/3}$ with $b_0 \geq a_1 N + a_1 + 1$, and the second inequality in \eqref{eq:beta-constraints} is satisfied by any $t \geq s + b_1 \beta^{1/3}$ with $b_1 \geq 64 n^2 N^2 + 16nN + N + 2$. Thus, \eqref{eq:exceptional-energy} holds in the $\beta > 0$ case with $a_2 := \max \, \{ b_0, b_1 \}$.
\end{proof}

Replacing $s$ with $s - a_2 \beta^{1/3}$---and increasing the constant by $1$ for the sake of symmetry in the two error terms---gives the following equivalent statement. 

\begin{cor} \label{cor:auxiliary}
    There exists a constant $a_3 \geq 1$, depending only on $N$ and $m$, such that the following holds. Let $(\Pi_{\bm{\lambda}})_{\bm{\lambda} \in U}$ be a family of generalized projections that is both transversal and strongly regular of order $\beta$. If $\mu \in \mathcal{M}(\Omega)$ is finite and has finite $s$-energy for some $m < s < N + m$, then
    \begin{equation*}
        \hdim \left\{ \bm{\lambda} \in U \!: \int_{\R^m} I_{s - m - a_3 \beta^{1/3}}(\mu_{\bm{\lambda},\mathbf{x}}) \, d\mathbf{x} = \infty \right\} \leq N + m - s + a_3 \beta^{1/3}.
    \end{equation*}
\end{cor}

Although this statement is perhaps easier to read, a direct application of Lemma \ref{lem:slicing-main} is more convenient for the purposes of proving the main results in the next section.

\newpage

\section{Proof of Theorem \texorpdfstring{\ref{thm:main}}{1.2}} \label{s:main-proof}

\subsection{The \texorpdfstring{``weak"}{"weak"} version} \label{ss:weak} First we prove a version of Theorem \ref{thm:main} in which the set $E \subset U$ contains the exceptional parameters for a fixed $A \subseteq \Omega$ but not necessarily for all the positive-$\mathcal{H}^s$-measure subsets of $A$. As a result, though, the $0 < \mathcal{H}^s(A) < \infty$ hypothesis can be relaxed to $\hdim A \geq s$ when $\beta > 0$.

\begin{thm} \label{thm:weak-main}
    There exists a constant $a \geq 1$, depending only on $N$ and $m$, such that the following holds. Let $\big( \Pi_{\bm{\lambda}} \!: \Omega \to \R^m \big)_{\bm{\lambda} \in U}$ be a family of generalized projections that is both transversal and strongly regular of degree $\beta \in [0,1)$. If $m < s < N+m$, then for every Borel set $A \subseteq \Omega$ with $\mathcal{H}^s(A) > 0$, there exists a set $F \subset U$ with
    \begin{equation} \label{eq:weak-main}
        \hdim F \leq N + m - s + a \beta^{1/3}
    \end{equation}
    such that, for every $\bm{\lambda} \in U \setminus F$, the inequality
    \begin{equation*}
        \hdim\!\big(A \cap \Pi_{\bm{\lambda}}^{-1}(\mathbf{x})\big) \geq (1-\beta) (s-m)
    \end{equation*}
    holds for a positive-Lebesgue-measure set of $\mathbf{x} \in \R^m$. Moreover, if $\beta > 0$, the hypothesis $\mathcal{H}^s(A) > 0$ can be relaxed to $\hdim A \geq s$.
\end{thm}

\begin{proof}
We begin with the case $\beta = 0$. With $A \subseteq \Omega$ as above, there exists an $s$-Frostman measure $\mu \in \mathcal{M}(A)$, so $I_\sigma(\mu) < \infty$ for all $\sigma < s$. Thus,
\begin{equation*}
    \hdim \big\{ \bm{\lambda} \in U \!: \Pi_{\bm{\lambda}\sharp}\mu \not\!\ll \mathcal{L}^m \big\} \leq N + m - \sup_{\sigma < s} \sigma = N + m - s
\end{equation*}
by Theorem \ref{thm:peres-schlag}, which together with Lemma \ref{lem:slicing-main} implies that
\begin{equation*}
    \hdim \underbrace{\left\{ \bm{\lambda} \in U \!: \int_{\R^m} I_{\sigma-m}(\mu_{\bm{\lambda},\mathbf{x}}) \, d\mathbf{x} = \infty \ \forall \+ s > \sigma > s - \eps \text{ or } \Pi_{\bm{\lambda}\sharp}\mu \not\!\ll \mathcal{L}^m \right\}}_{{\displaystyle \raisebox{-0.2cm}{$F_\eps$}}} \leq N + m - s + \eps.
\end{equation*}
Now, by the Lebesgue differentiation theorem and the definition of a slice measure,
\begin{equation} \label{eq:sliced-measures-exist}
    \mathcal{L}^m\big( \underbrace{\big\{ \mathbf{x} \in \R^m \!: \mu_{\bm{\lambda},\mathbf{x}} \text{ is defined and } \mu_{\bm{\lambda},\mathbf{x}} \not\equiv 0 \big\}}_{{\displaystyle N_{\bm{\lambda}}}} \big) > 0
\end{equation}
for all $\bm{\lambda} \in U$ with $\Pi_{\bm{\lambda}\sharp}\mu \not\!\ll \mathcal{L}^m$. Whenever $\bm{\lambda} \in U \setminus F_\eps$ and $\mathbf{x} \in N_{\bm{\lambda}}$, the inequality $I_{\sigma-m}(\mu_{\bm{\lambda},\mathbf{x}}) < \infty$ holds for all $s - \eps < \sigma < s$; thus, for all $\bm{\lambda} \in U \setminus F_\eps$, it follows from the H\"{o}lder lower bound \eqref{eq:bi-holder}---which is a \textit{Lipschitz} bound when $\beta = 0$---that
\begin{align*}
    \hdim\!\big( A \cap \Pi_{\bm{\lambda}}^{-1}(\mathbf{x}) \big) &= \hdim\!\big( \Psi_{\bm{\lambda}}(A) \cap \big\{ \mathbf{z} \in \R^{m+n} \!: \mathbf{z}^{(1)} = \mathbf{x} \big\} \big) \\
    & \geq \hdim(\spt \mu_{\bm{\lambda},\mathbf{x}}) \geq \sup_{\sigma < s} \sigma - m = s - m
\end{align*}
for $\mathcal{L}^m$-positively-many $\mathbf{x} \in \R^m$. Hence,
\begin{equation*}
    \mathcal{L}^m\Big( \big\{ \mathbf{x} \in \R^m \!: \hdim\!\big( A \cap \Pi_{\bm{\lambda}}^{-1}(\mathbf{x}) \big) \geq s - m \big\} \Big) \geq \mathcal{L}^m(N_{\bm{\lambda}}) > 0,
\end{equation*}
so $\bm{\lambda}$ does \textit{not} belong to the exceptional set
\begin{equation*}
    F := \Big\{ \bm{\lambda} \in U \!: \mathcal{L}^m\Big( \big\{ \mathbf{x} \in \R^m \!: \hdim\!\big( A \cap \Pi_{\bm{\lambda}}^{-1}(\mathbf{x}) \big) \geq s - m \big\} \Big) = 0 \Big\};
\end{equation*}
that is, $F \subseteq F_\eps$, whence $\hdim F \leq N + m - s + \eps$ for all $\eps > 0$.

Now let
\begin{equation*}
    a := \max \big\{ (N+m)a_0, a_2 \big\},
\end{equation*}
suppose $\beta > 0$, and suppose $\hdim A \geq s$ (but not necessarily $\mathcal{H}^s(A) > 0$) for some $m < s < N + m - a \beta^{1/3}$. (If $s \geq N + m - a \beta^{1/3}$, then \eqref{eq:weak-main} is trivial.) As in the $\beta = 0$ case, for any $\mu \in \mathcal{M}(A)$ with finite $(s - \beta^{1/3})$-energy, the inequality
\begin{equation*}
    \hdim \underbrace{\left\{ \bm{\lambda} \in U \!: \int_{\R^m} I_{s-m}(\mu_{\bm{\lambda},\mathbf{x}}) \, d\mathbf{x} = \infty \text{ or } \Pi_{\bm{\lambda}\sharp}\mu \not\!\ll \mathcal{L}^m \right\}}_{{\displaystyle \raisebox{-0.2cm}{$\medtilde{F}$}}} \leq N + m - s + a \beta^{1/3}
\end{equation*}
holds by Theorem \ref{thm:peres-schlag}, Lemma \ref{lem:slicing-main}, and our choice of $a$. For all $\bm{\lambda} \in U \setminus \medtilde{F}$, \eqref{eq:sliced-measures-exist} remains valid as in the $\beta = 0$ case, and for these $\bm{\lambda}$ and all $\mathbf{x} \in N_{\bm{\lambda}}$, the upper H\"{o}lder condition \eqref{eq:bi-holder} on $\Psi_{\bm{\lambda}}$ entails that
\begin{align*}
    \hdim\!\big( A \cap \Pi_{\bm{\lambda}}^{-1}(\mathbf{x}) \big) &\geq (1-\beta) \hdim \big\{ \mathbf{y} \in \R^n \!: (\mathbf{x},\mathbf{y}) \in \Psi_{\bm{\lambda}}(A) \big\} \\
    &\geq (1-\beta) \hdim(\spt \mu_{\bm{\lambda},\mathbf{x}}) \\
    &\geq (1-\beta) (s-m).
\end{align*}
This again implies that $\bm{\lambda}$ does not belong to the exceptional set
\begin{equation*}
    F := \Big\{ \bm{\lambda} \in U \!: \mathcal{L}^m\Big( \Big\{ \mathbf{x} \in \R^m \!: \hdim\!\big( A \cap \Pi_{\bm{\lambda}}^{-1}(\mathbf{x}) \big) \geq (1-\beta) (s-m) \Big\} \Big) = 0 \Big\},
\end{equation*}
so $F \subseteq \medtilde{F}$ and $\hdim F \leq N + m - s + a \beta^{1/3}$.
\end{proof}

\vspace{-0.33cm} \subsection{The full version} \label{ss:full} At last we complete the proof of Theorem \ref{thm:main}. The trick is identical to that of Falconer and Mattila \cite{falconer2016strong}: rerun the proof of the ``basic" version---in our case, Theorem \ref{thm:weak-main}---with a countable family of $s$-Frostman measures $\mu_j$ such that $\sum_{j=0}^\infty \mu_j$ is the restriction of $\mathcal{H}^s$ to $A$.

\begin{proof}[\define{Proof of Theorem \ref{thm:main}}.]
Since $\mathcal{H}^s(A) < \infty$, the upper $s$-density
\begin{equation*}
    \Theta^{*s}(A,\bm{\omega}) := \limsup_{r \to 0} \frac{\mathcal{H}^s(A \cap B(\bm{\omega},r))}{(2r)^s}
\end{equation*}
is less than or equal to $1$ at $\mathcal{H}^s$-a.e.\! $\bm{\omega} \in A$. For each $r > 0$, the function $\bm{\omega} \mapsto \mathcal{H}^s(A \cap B(\bm{\omega},r))$ is upper semicontinuous (and, hence, $\bm{\omega} \mapsto \Theta^{*s}(A,\bm{\omega})$ is Borel measurable), so the set of all $\bm{\omega} \in A$ with upper density at most $1$ can we written as a countable union of (countable intersections of) Borel sets:
\begin{equation*}
    A' := \big\{ \bm{\omega} \in A \!: \Theta^{*s}(A,\bm{\omega}) \leq 1 \big\} = \bigcup_{j=1}^\infty \bigcap_{r \in \Q_+} \big\{ \bm{\omega} \in A \!: \mathcal{H}^s(A \cap B(\bm{\omega},r)) \leq r^s \big\}.
\end{equation*}
Thus, with $A_1 := \big\{ \bm{\omega} \in A \!: \mathcal{H}^s(A \cap B(\bm{\omega},r)) \leq r^s \ \forall \+ r \in \Q_+ \big\}$,
\begin{equation*}
     A_j := \left\{ x \in A \setminus \bigcup_{i=1}^{j-1} A_i \!: \, \mathcal{H}^t(A \cap B(\bm{\omega},r)) \leq j \+ r^s \ \, \forall \+ r \in \Q_+ \right\} \quad \text{for all } j \geq 2,
\end{equation*}
and $\mu_j := \mathcal{H}^s \restrict A_j$, we have
\begin{equation*}
    A' = \bigsqcup_{j=1}^\infty A_j \quad \text{and, consequently,} \quad \mathcal{H}^s \restrict A' = \mathcal{H}^s \restrict A = \sum_{j=0}^\infty \mu_j.
\end{equation*}
Moreover, each $\mu_j$ that is not identically $0$ is $s$-Frostman (with coefficient $j$), so every $\mu_j$ has finite $\sigma$-energy for all $0 \leq \sigma < s$. As in the proof of Theorem \ref{thm:weak-main},
\begin{align*}
    \hdim \underbrace{\left\{ \bm{\lambda} \in U \!: \int_{\R^m} I_{(1-\beta)(\sigma - m)}(\mu_{j,\bm{\lambda},\mathbf{x}}) \, d\mathbf{x} = \infty \text{ or } \Pi_{\bm{\lambda}\sharp}\mu_j \not\!\ll \mathcal{H}^m \right\}}_{{\displaystyle \raisebox{-0.2cm}{$E_{\sigma,j}$}}} \leq N + m - \sigma + a \beta^{1/3},
\end{align*}
where $\mu_{j,\bm{\lambda},\mathbf{x}} := (\mu_j)_{\bm{\lambda},\mathbf{x}}$.

Now let $B \subseteq A$ be a Borel set with $\mathcal{H}^s(B) > 0$. Then there exists some $j \in \Z_+$ such that $\mu_j(B) > 0$. Continuing to follow the proof of Theorem \ref{thm:weak-main} with $\mu_j \restrict B$ in place of $\mu$, one sees that, for all $\bm{\lambda} \in U \setminus E_{\sigma,j}$,
\begin{equation*}
    \hdim\!\big( B \cap \Pi_{\bm{\lambda}}^{-1}(\mathbf{x}) \big) \geq \sigma - m - a \beta^{1/3}
\end{equation*}
for a positive-$\mathcal{L}^m$-measure set of $\mathbf{x} \in \R^m$, namely, for all $\mathbf{x} \in \R^m$ such that $\mu_{j,\bm{\lambda},\mathbf{x}} \restrict B \not\equiv 0$ is defined and $I_{\sigma - m - a \beta^{1/3}}(\mu_{j,\bm{\lambda},\mathbf{x}})$ is finite. It follows that, for all Borel sets $B \subseteq A$ with $\mathcal{H}^s(B) > 0$,
\begin{equation*}
    \Big\{ \bm{\lambda} \in U \!: \mathcal{L}^m\Big( \Big\{ \mathbf{x} \in \R^m \!: \hdim\!\big( B \cap \Pi_{\bm{\lambda}}^{-1}(\mathbf{x}) \big) \geq \sigma - m - a \beta^{1/3} \Big\} \Big) > 0 \Big\} \subseteq \bigcup_{j=1}^\infty E_{\sigma,j}.
\end{equation*}
In the usual way, taking the union over $\sigma = s - \tfrac{1}{k}, \+ s - \tfrac{1}{k+1}, \+ \dots$ and the intersection over all $k \in \Z_+$ gives
\begin{equation*}
    \bigcup_{\substack{B \subseteq A \, \text{Borel} \\ \mathcal{H}^s(B) > 0}} \- \Big\{ \bm{\lambda} \in U \!: \mathcal{L}^m\-\Big( \- \Big\{ \mathbf{x} \in \R^m \!: \hdim\!\big( B \cap \Pi_{\bm{\lambda}}^{-1}(\mathbf{x}) \big) \geq s - m - a \beta^{1/3} \Big\} \- \Big) \! > \! 0 \Big\} \subseteq \bigcup_{j=1}^\infty E_{s,j},
\end{equation*}
where $\hdim \bigcup_{j=1}^\infty E_{s,j} \leq N + m - s$ by countable stability. This concludes the proof.
\end{proof}

\section{Applications} \label{s:applications}

\subsection{Vertical sections of the Heisenberg group} Equip $\R^{2n}$ with the standard symplectic form $\omega \!: \R^{2n} \times \R^{2n} \to \R$ and let $G_h(n,m) \subseteq \Gr(2n,m)$ be the \textit{isotropic Grassmannian}---the collection of $m$-dimensional subspaces $V \subseteq \R^{2n}$ such that $\omega(V,V) = \{0\}$ (the \textit{isotropic subspaces} of $\R^{2n}$). It is well-known that $G_h(n,m)$ is a Riemannian submanifold of $\Gr(2n,m)$ of dimension $2nm - \tfrac{m(3m-1)}{2}$. By Hovila \cite{hovila2014transversality} Theorem 1.1, the family $(\pi_{\scalebox{0.65}{$V$}})_{V \in G_h(n,m)}$ of orthogonal projections is transversal and strongly regular of degree $0$. (Hovila only claims that \eqref{eq:strong-reg} holds through $|\bm{\alpha}| = 2$, but it is clear that it also holds for all $\bm{\alpha}$. In fact, strong regularity is inherited directly from $(\pi_{\scalebox{0.65}{$V$}})_{V \in \Gr(2n,m)}$, so the only novelty is that only the transversality condition \eqref{eq:transvers} also holds.)

The \textit{$n$\textsuperscript{th} Heisenberg group} $\bbh^n$ is a Lie group that, under a pointwise identification $\bbh^n \cong \R^{2n} \times \R$, carries the group operation
\begin{equation*}
    (\mathbf{z},\tau) \ast (\mathbf{w},\sigma) := (\mathbf{z} + \mathbf{w}, \tau + \sigma + 2 \+ \omega(\mathbf{z},\mathbf{w})).
\end{equation*}
With the \textit{Heisenberg metric}
\begin{equation*}
    d_{\mathrm H}\big( (\mathbf{z},\tau), (\mathbf{w},\sigma) \big) := \big\| (\mathbf{z},\tau)^{-1} \ast (\mathbf{w},\sigma) \big\|_{\mathrm H}, \quad \| (\mathbf{z},\tau) \|_{\mathrm H} := \big( \| \mathbf{z} \|^4 + \tau^2 \big)^{1/4},
\end{equation*}
$\bbh^n$ has Hausdorff dimension $2n+2$. Roughly, this is because the Heisenberg metric $d_{\mathrm H}$ restricted to the $(2n+1)$\textsuperscript{st} coordinate axis is the square root of the Euclidean metric on $\R$, so the contributions to dimension of a set owing to its extent in this vertical direction are doubled.

Call a subgroup $\mathbb{V} := V \times \{0\} \subset \bbh^n$ with $V \in G_h(n,m)$ a \textit{horizontal subgroup} and define the \textit{horizontal projection} $\pi_\mathbb{V} \!: \bbh^n \to \mathbb{V}$ by $\pi_\mathbb{V}(\mathbf{z},\tau) := (\pi_{\scalebox{0.65}{$V$}}(\mathbf{z}),0)$. Since ${d_{\mathrm H} |}_\mathbb{V}$ is precisely the Euclidean metric on $\mathbb{V}$, such a subgroup has Hausdorff dimension $m$, whereas the fibers of $\pi_\mathbb{V}$ have dimension $2n - m + 2$ rather than $2n - m + 1$. The isotropic Grassmannian parametrizes the horizontal projections in $\bbh^n$ via $V \mapsto \pi_\mathbb{V}$, but $(\pi_\mathbb{V})_{V \in G_h(n,m)}$ is not itself a transversal family. However, natural mapping of $\bbh^n$ onto $\R^{2n}$ cannot decrease the dimension of a set by more than $2$, for which reason Hovila's transversality result nevertheless allows for the application of Theorem \ref{thm:main} to horizontal projections.

\begin{cor} \label{cor:heisenberg}
    For every Borel set $A \subseteq \R^{2n}$ with $0 < \mathcal{H}^s(A) < \infty$ for some $s > m$,
    \begin{align*}
        &\hdim \bigcup_{\substack{B \subseteq A \, \text{Borel} \\ \mathcal{H}^s(B) > 0}} \Big\{ V \in G_h(n,m) \!: \mathcal{H}^m\big( \big\{ \mathbf{x} \in V \!: \hdim \! \big( \pi_{\scalebox{0.65}{$V$}}^{-1}(\mathbf{x}) \cap B \big) = s - m \big\} \big) = 0 \Big\} \\
        &\hspace{9.67cm} \leq 2nm - \frac{m(3m+1)}{2} - s.
    \end{align*}
    Consequently, if $A' \subseteq \bbh^n$ is Borel with $0 < \mathcal{H}^s(A) < \infty$ for some $s > m+2$, then
    \begin{align*}
        &\hdim \bigcup_{\substack{B' \subseteq A' \, \text{Borel} \\ \mathcal{H}^s(B') > 0}} \Big\{ V \in G_h(n,m) \!: \mathcal{H}^m\big( \big\{ \mathbf{p} \in V \!: \hdim \! \big( \pi_\mathbb{V}^{-1}(\mathbf{p}) \cap B' \big) = s - m \big\} \big) = 0 \Big\} \\
        &\hspace{9.33cm} \leq 2nm - \frac{3m(m-1)}{2} - s + 2.
    \end{align*}
\end{cor}

Theorem \ref{thm:main} states only that the set
\begin{equation*}
    \big\{ \mathbf{x} \in V \!: \hdim \! \big( \pi_{\scalebox{0.65}{$V$}}^{-1}(\mathbf{x}) \cap B \big) \geq s - m \big\}
\end{equation*}
has positive $\mathcal{H}^m$-measure for all but a few $V \in G_h(n,m)$, whereas the set
\begin{equation*}
    \big\{ \mathbf{x} \in V \!: \hdim \! \big( \pi_{\scalebox{0.65}{$V$}}^{-1}(\mathbf{x}) \cap B \big) > s - m \big\}
\end{equation*}
\textit{never} has positive $\mathcal{H}^m$-measure (see the discussion following Theorem \ref{thm:ortho}), for which reason we are guaranteed that, generically, $\hdim \! \big( \pi_{\scalebox{0.65}{$V$}}^{-1}(\mathbf{x}) \cap B \big)$ is \textit{exactly} $s-m$ for a positive-measure set of $\mathbf{x} \in V$.

Since $\pi_\mathbb{V}^{-1}(\mathbf{p}) = \mathbb{V}^\perp \ast \mathbf{p}$, Corollary \ref{cor:heisenberg} is a statement about the Hausdorff dimension of slices by right cosets of the \textit{vertical subgroups} of the Heisenberg group. This improves on \cite{balogh2012projection} Theorem 1.5 concerning slices in the Heisenberg group.

\phantomsection
\section*{Appendix. Proof of Lemma \texorpdfstring{\ref{lem:technical}}{3.3}} \label{s:appendix}
\renewcommand{\theprop}{A.\arabic{prop}}
\renewcommand{\theequation}{A.\arabic{equation}}

Before beginning the proof of Lemma \ref{lem:technical}, let us take a moment to understand why the steps that follow are necessary. The analogue of Lemma \ref{lem:technical} in Peres--Schlag \cite{peres2000smoothness} is Lemma 7.10, in which the dependence on $D_{\bm{\lambda}} \Phi_{\bm{\lambda}}(\bm{\omega},\bm{\zeta})$ is not present. This alone could be rectified by assuming higher-order regularity on $(\Pi_{\bm{\lambda}})_{\bm{\lambda} \in U}$, and for us virtually no modification would be necessary since we assume that the family is $C^\infty$ in $\bm{\lambda}$. The factor of $2^{j-i}$ is the more serious issue, since the dependence on $i$ is not present in \cite{peres2000smoothness} and the right-hand side of \eqref{eq:technical} also does not depend on $i$. Reworking the computation of Lemma 7.10 to verify that the resulting estimate does not depend on $i$ is all that needs to be done, and although one finds that the only necessary changes are essentially notational, we follow Orponen's example (\cite{orponen2014slicing} Lemma 4.2) and write out the details here in the interest of completeness.

First we require a somewhat cumbersome statement about the distribution of critical points of $\Phi_{\bm{\lambda}}(\bm{\omega},\bm{\zeta})$ for fixed, distinct $\bm{\omega},\bm{\zeta} \in \Omega$. It takes the form of an implicit function theorem that, leveraging transversality and regularity, establishes quantitative bounds on the size and number of balls in $U$ required for the implicit functions and their inverses to be bi-H\"{o}lder on each ball with the same H\"{o}lder constant across balls.

Peres and Schlag phrase the result in terms of what they call \textit{$L,\delta$-regularity}; the terminology suppresses the dependence on $\beta$, so we refer to this as \textit{$L,\delta,\beta$-regularity} for clarity. If $(\Pi_{\bm{\lambda}})_{\bm{\lambda} \in U}$ is $\infty,\delta,\beta$-regular for some $\delta \in [0,1)$, then it is $\infty,\delta,\beta$-regular for \textit{all} $\delta \in [0,1)$. Rather than transcribe \cite{peres2000smoothness} Lemma 7.7 verbatim, we avoid $L,\delta,\beta$-regularity and state the conclusion with $\delta = 1$. Peres and Schlag do not define $L,1,\beta$-regularity (doing so would not improve the statements of their main results), but their proof of Lemma 7.7 works with no modifications when $\delta = 1$ (see also the proof of \cite{orponen2014slicing} Lemma A.1). We therefore state their result with $\delta = 1$ as Lemma \ref{lem:crit-pts} below and refer the reader to \cite{peres2000smoothness} for the proof.

As in previous sections, we assume throughout this appendix that $(\Pi_{\bm{\lambda}})_{\bm{\lambda} \in U}$ is transversal and regular of order $\beta \in [0,1)$. All implicit constants are assumed to depend on this family. Without loss of generality, each compact set $K \subset U$ referred to below is assumed to be the closure of its interior. The symbol $\delta = \delta_{\beta,K} > 0$ denotes the implicit constant in \eqref{eq:transvers} of Definition \ref{defn:gen-proj}, i.e., a number such that
\begin{equation*}
    |\Phi_{\bm{\lambda}}(\bm{\omega},\bm{\zeta})| \leq \delta \+ d(\bm{\omega},\bm{\zeta})^\beta \quad \Longrightarrow \quad \det\!\left[ D_{\bm{\lambda}} \Phi_{\bm{\lambda}}(\bm{\omega},\bm{\zeta}) \big( D_{\bm{\lambda}} \Phi_{\bm{\lambda}}(\bm{\omega},\bm{\zeta}) \big)^{\-\mathrm{t}} \right] \geq \delta^2 d(\bm{\omega},\bm{\zeta})^{2\beta}.
\end{equation*}
In all that follows, $\bm{\omega},\bm{\zeta} \in \Omega$ are fixed and $r := d(\bm{\omega},\bm{\zeta})$.

Finally, we ask the reader's forgiveness for our choices of notation in this appendix. Symbols will frequently be recycled without mention, but we hope that the intended meanings are always clear from context.

\begin{lem}[Peres--Schlag \cite{peres2000smoothness} Lemma 7.7] \label{lem:crit-pts}
    Let $K \subseteq U$ be a compact set. There exist constants $C_0,C_1,C_2 > 0$ depending only on $K$ such that, for all distinct $\bm{\omega},\bm{\zeta} \in \Omega$, there exist $\bm{\lambda_1}, \dots, \bm{\lambda_M} \in U$ with $M \leq C_2 r^{-(m+2)N}$ such that
    \begin{align*}
        &\big\{ \bm{\lambda} \in K \!: |\Phi_{\bm{\lambda}}(\bm{\omega},\bm{\zeta})| \leq C_0 r^{(m+2)\beta} \big\} \subseteq \bigcup_{k=1}^M B\big( \bm{\lambda_k}, C_1 r^{(m+2)\beta} \big) \quad \text{and} \\
        &\bigcup_{k=1}^M \underbrace{B\big( \bm{\lambda_j}, 2C_1 r^{(m+2)\beta} \big)}_{{\displaystyle B_k}} \subseteq \Big\{ \bm{\lambda} \in U \!: |\Phi_{\bm{\lambda}}(\bm{\omega},\bm{\zeta})| \leq \delta r^\beta \Big\}.
    \end{align*}
    Moreover, for each $k \in \{1,\dots,M\}$ there exist coordinate directions $1 \leq i_1 \leq \cdots i_{N-m} \leq N$ such that, for all $\bm{\kappa} \in \R^{N-m}$, the mapping $F_{\bm{\kappa}}$ defined by restricting $\bm{\lambda} \mapsto \Phi_{\bm{\lambda}}(\bm{\omega},\bm{\zeta})$ to the set $\big\{ \bm{\lambda} \in B_k \!: \lambda_{i_1} = \kappa_1, \ \dots, \ \lambda_{i_{N-m}} = \kappa_{N-m} \big\}$ is a diffeomorphism onto its image that satisfies
    \begin{equation} \label{eq:diffeo-deriv-bounds}
        \Big| \det\-\Big[ \big( DF_{\bm{\kappa}} \big)^{-1} \Big] \Big| \leq C_2 r^{-\beta} \quad \text{and} \quad \Big\| \big( DF_{\bm{\kappa}} \big)^{-1} \Big\| \leq C_2 r^{-m\beta}.
    \end{equation}
\end{lem}

Denote $\Phi_{\bm{\lambda}} := \Phi_{\bm{\lambda}}(\bm{\omega},\bm{\zeta})$ and recall what we intend to show:

\begin{nlem}[Lemma \ref{lem:technical}, Restated]
    There exists a constant $a_1 \geq 1$ (depending only on $N$ and $m$) such that, for all $\rho \in C_c^\infty(U)$, all $c \geq 0$, and all $\bm{\omega},\bm{\zeta} \in \Omega$,
    \begin{equation*}
        \left| \+ \int_U \rho(\bm{\lambda}) \+ \widehat{\eta^{(1)}}\big( 2^j r \+ \Phi_{\bm{\lambda}}(\bm{\omega},\bm{\zeta}) \big) \widehat{\eta^{(2)}}\big( 2^{j-i} r \+ D_{\bm{\lambda}} \Phi_{\bm{\lambda}}(\bm{\omega},\bm{\zeta}) \big) \+ d\bm{\lambda} \+ \right| \lesssim_{\rho,c} \, \big( 1 + 2^j r^{1 + a_1 \beta} \big)^{-c}.
    \end{equation*}
\end{nlem}

As before, $\eta$ is a smooth bump function compactly supported in $(0,\infty)$. The support of $\rho$ will be denoted $K$, and we assume that $\rho \not\equiv 0$.

\begin{proof}[\define{Proof of Lemma \ref{lem:technical}}]
Let
\begin{equation*}
    \phi(\mathbf{x}) := \widehat{\eta^{(1)}}(\mathbf{x}), \quad \psi(\mathbf{y}) := \widehat{\eta^{(2)}}(\mathbf{y}), \quad \text{and} \quad \mathrm{R}(\bm{\lambda}) := \rho(\bm{\lambda}) \+ \psi\big( 2^{j-i} r \+ D_{\bm{\lambda}} \Phi_{\bm{\lambda}} \big).
\end{equation*}
Suppressing the dependence of implicit constants on $\rho$, we write the desired inequality as
\begin{equation*}
    \left| \+ \int_K \mathrm{R}(\bm{\lambda}) \phi\big( 2^j r \+ \Phi_{\bm{\lambda}} \big) \+ d\bm{\lambda} \+ \right| \lesssim_c \big( 1 + 2^j r^{1 + a_1 \beta} \big)^{-c}.
\end{equation*}
If $2^j r^{1 + a_1' \beta} \leq 1$, where $a_1' \geq 1$ is fixed but yet-to-be-specified, then this is immediate: both $\phi$ and $\psi$ are bounded, so
\begin{align*}
    \left| \+ \int_K \mathrm{R}(\bm{\lambda}) \phi\big( 2^j r \+ \Phi_{\bm{\lambda}} \big) \+ d\bm{\lambda} \+ \right| &\leq \| \rho \|_{L^1(U)} \| \psi \|_{L^\infty(\R^n)} \| \phi \|_{L^\infty(\R^m)} \\
    &\lesssim 1 \sim_c 2^{-c} \leq \big( 1 + 2^j r^{1 + a_1 \beta} \big)^{-c}.
\end{align*}
Assume, therefore, that $2^j r^{1 + a_1' \beta} > 1$ and let $f \in C^\infty(\R)$ be a bump function for $[-1,1]$ supported in $[-2,2]$. Then
\begin{align}
    & \int_K \mathrm{R}(\bm{\lambda}) \phi\big( 2^j r \+ \Phi_{\bm{\lambda}} \big) \+ d\bm{\lambda} \nonumber \\
    &\hspace{1cm} = \int_K \mathrm{R}(\bm{\lambda}) \phi\big( 2^j r \+ \Phi_{\bm{\lambda}} \big) f\big( \tfrac{1}{\delta r^\beta} \Phi_{\bm{\lambda}} \big) \+ d\bm{\lambda} \label{eq:app-line-1} \\
    &\hspace{1.4cm} + \int_K \mathrm{R}(\bm{\lambda}) \phi\big( 2^j r \+ \Phi_{\bm{\lambda}} \big) \big[ 1 - f\big( \tfrac{1}{\delta r^\beta} \Phi_{\bm{\lambda}} \big) \big] \+ d\bm{\lambda}, \label{eq:app-line-2}
\end{align}
so we can bound \eqref{eq:app-line-1} and \eqref{eq:app-line-2} individually. The second integral is straightforward to handle: $|\Phi_{\bm{\lambda}}| \geq \delta r^\beta$ on the support of the integrand by the definition of $f$, and since $\phi$ is a Schwartz function,
\begin{equation*}
    \big| \phi\big( 2^j r \+ \Phi_{\bm{\lambda}} \big) \big| \lesssim_c \big( 1 + \delta 2^j r \+ \Phi_{\bm{\lambda}} \big)^{-c} \leq \big( 1 + \delta 2^j r^{1+\beta} \big)^{-c} \lesssim_{|\Omega|,\beta,c} \big( 1 + 2^j r^{1+a_1\beta} \big)^{-c}.
\end{equation*}
Thus,
\begin{equation} \label{eq:app-line-3}
\begin{aligned}
    & \int_K \mathrm{R}(\bm{\lambda}) \phi\big( 2^j r \+ \Phi_{\bm{\lambda}} \big) \big[ 1 - f\big( \tfrac{1}{\delta r^\beta} \Phi_{\bm{\lambda}} \big) \big] \+ d\bm{\lambda} \lesssim \int_K \mathrm{R}(\bm{\lambda}) \big( 1 + 2^j r^{1+a_1\beta} \big)^{-c} \, d\bm{\lambda} \\
    &\qquad \leq \| \rho \|_{L^1(U)} \| \psi \|_{L^\infty(\R^n)} \big( 1 + 2^j r^{1+a_1\beta} \big)^{-c} \sim_c \big( 1 + 2^j r^{1+a_1\beta} \big)^{-c}.
\end{aligned}
\end{equation}
Our goal, then, is to bound \eqref{eq:app-line-1} by this same expression. Recalling Lemma \ref{lem:crit-pts}, let $(\chi_k)_{k=1}^M$ be a smooth partition of unity for $\big\{ \bm{\lambda} \in K \!: |\Phi_{\bm{\lambda}}| \leq C_0 r^{(m+2)\beta} \big\}$ subordinate to $(B_k)_{k=1}^M$. By the mean value theorem, we may choose these $\chi_k$ so that
\begin{equation} \label{eq:chi-k}
    \max_{1 \leq k \leq M} \big\| \partial^{\bm{\alpha}} \chi_k \big\|_{L^\infty(U)} \lesssim_{|\bm{\alpha}|} r^{-|\bm{\alpha}|(m+2)\beta},
\end{equation}
as each $B_k$ has radius $\sim \- r^{(m+2)\beta}$. Since $f\big( \tfrac{1}{\delta r^\beta} \Phi_{\bm{\lambda}} \big) = 1$ for all $\bm{\lambda} \in \supp \chi_k$ and all $1 \leq k \leq M$,
\begin{align}
    & \int_K \mathrm{R}(\bm{\lambda}) \phi\big( 2^j r \+ \Phi_{\bm{\lambda}} \big) f\big( \tfrac{1}{\delta r^\beta} \Phi_{\bm{\lambda}} \big) \+ d\bm{\lambda} = \sum_{k=1}^M \int_K \mathrm{R}(\bm{\lambda}) \phi\big( 2^j r \+ \Phi_{\bm{\lambda}} \big) \chi_k(\bm{\lambda}) \+ d\bm{\lambda} \label{eq:app-line-4} \\
    &\hspace*{3.75cm} + \int_K \mathrm{R}(\bm{\lambda}) \phi\big( 2^j r \+ \Phi_{\bm{\lambda}} \big) \Bigg[ 1 - \sum_{k=1}^M \chi_k(\bm{\lambda}) \Bigg] f\big( \tfrac{1}{\delta r^\beta} \Phi_{\bm{\lambda}} \big) \+ d\bm{\lambda}. \label{eq:app-line-5}
\end{align}
The second line admits a simple estimate in the same manner as \eqref{eq:app-line-2}: the inequality $|\Phi_{\bm{\lambda}}| \geq C_0 r^{(m+2)\beta}$ holds on the support of the integrand in \eqref{eq:app-line-5} by Lemma \ref{lem:crit-pts}, so by the same argument as that used to derive \eqref{eq:app-line-3} (with $C_0 r^{(m+2)\beta}$ in place of $\delta r^\beta$),
\begin{equation*}
    \int_K \mathrm{R}(\bm{\lambda}) \phi\big( 2^j r \+ \Phi_{\bm{\lambda}} \big) \Bigg[ 1 - \sum_{k=1}^M \chi_k(\bm{\lambda}) \Bigg] f\big( \tfrac{1}{\delta r^\beta} \Phi_{\bm{\lambda}} \big) \+ d\bm{\lambda} \sim_c \big( 1 + 2^j r^{1+a_0\beta} \big)^{-c}.
\end{equation*}
It therefore remains to bound \eqref{eq:app-line-4}. Fix $1 \leq k \leq M$ and recall from Lemma \ref{lem:crit-pts} that, for all $\bm{\lambda}'' \in \R^{N-m}$, the map $\bm{\lambda}' \mapsto F_{\bm{\lambda}''}(\bm{\lambda}') := \Phi_{(\bm{\lambda}',\bm{\lambda}'')}$ is a diffeomorphism of $\{ \bm{\lambda}' \in \R^m \!: (\bm{\lambda}',\bm{\lambda}'') \in B_k \}$ onto its image $B_k'$, possibly after a permutation of the coordinate axes. Henceforth denoting $\bm{\lambda} =: (\bm{\lambda}',\bm{\lambda}'')$, we assume that our smooth function $\rho$ has the form $\rho(\bm{\lambda}) = \rho_1(\bm{\lambda}') \rho_2(\bm{\lambda}'')$; since functions of this form constitute a basis for both $L^1(K)$ and $C_c(K)$, this poses no loss of generality. With $1 \leq k \leq M$ fixed, let
\begin{equation} \label{eq:g-lambda-double-prime}
    g_{\bm{\lambda}''}(\mathbf{u}) := \rho_1\big( F_{\bm{\lambda}''}^{-1}(\mathbf{u}) \big) \chi_k\big( F_{\bm{\lambda}''}^{-1}(\mathbf{u}) \big) h_{\bm{\lambda}''}(\mathbf{u}) \big| \- \det D_\mathbf{u} F_{\bm{\lambda}''}^{-1}(\mathbf{u}) \big|,
\end{equation}
for all $\mathbf{u} \in B_k'$, where
\begin{align*}
    h_{\bm{\lambda}''}(\mathbf{u}) :\!\- &= \psi\Big( 2^{j-i} r \Big[ \, \big( D_\mathbf{u} F_{\bm{\lambda}''}^{-1}(\mathbf{u}) \big)^{-1} \ \+ D_{\bm{\lambda}''} \Phi_{(\bm{\lambda}',\bm{\lambda}'')} \big|_{\bm{\lambda}' = F_{\bm{\lambda}''}^{-1}(\mathbf{u})} \, \Big] \Big) \\
    &= \psi\Big( 2^{j-i} r \Big[ \, D_{F_{\bm{\lambda}''}^{-1}(\mathbf{u})} \mathbf{u} \ D_{\bm{\lambda}''} \mathbf{u} \, \Big] \Big).
\end{align*}
We use the change of variables $\bm{\lambda}' = F_{\bm{\lambda}''}^{-1}(\mathbf{u})$ (equivalently, $\mathbf{u} = \Phi_{\bm{\lambda}}$) to write
\begin{align}
    & \int_K \mathrm{R}(\bm{\lambda}) \phi\big( 2^j r \+ \Phi_{\bm{\lambda}} \big) \chi_k(\bm{\lambda}) \+ d\bm{\lambda} \nonumber \\
    &\qquad = \int_K \rho_1(\bm{\lambda}') \chi_k(\bm{\lambda}) \phi\big( 2^j r \+ \Phi_{\bm{\lambda}} \big) \psi\big( 2^{j-i} r \+ D_{\bm{\lambda}} \Phi_{\bm{\lambda}} \big) \rho_2(\bm{\lambda}'') \+ d\bm{\lambda} \nonumber \\
    &\qquad = \int_{\R^{N-m}} \rho_2(\bm{\lambda}'') \int_{\R^m} g_{\bm{\lambda}''}(\mathbf{u}) \phi(2^j r\mathbf{u}) \, d\mathbf{u} \+ d\bm{\lambda}''. \label{eq:final-integral}
\end{align}
To estimate this integral, we will bound the inner integral uniformly in $\bm{\lambda}''$ via a Taylor expansion and use the fact that $\rho_2$ is Schwartz to handle the outer integral. For this reason, we require two general differentiation formulas. The first concerns the derivatives of inverse functions:

\begin{lem}[Peres--Schlag \cite{peres2000smoothness} Lemma 7.9]
    Let $V \subseteq \R^m$ be an open set and let $F \!: V \to \R^m$ be a $C^L$-diffeomorphism onto its image. For all $\bm{\alpha} \in \N^m$ with $|\bm{\alpha}| = L$,
    \begin{equation*}
        \partial^{\bm{\alpha}}(F^{-1}) = \sum_{k=0}^{L-1} \sum_{j=1}^{L-1} \sum_{\bm{\sigma},\mathbf{q}} (DF \circ F^{-1})^{-L-j} \mathbf{v}_{\bm{\alpha},j}(\bm{\sigma},\mathbf{q}) \prod_{i=1}^j \big( \partial^{\bm{\sigma}_i} F_{q_i} \big) \circ F^{-1},
    \end{equation*}
    where:
    \begin{enumerate}[label={\normalfont\textbullet},topsep=0pt,itemsep=0pt]
        \item the third sum runs over all $\mathbf{q} \in \{1,\dots,m\}^j$ and over all $\bm{\sigma} \in \{2,\dots,L\}^{m \times j}$ (i.e., all $j$-tuples consisting of multi-indices $\bm{\sigma}_1,\dots,\bm{\sigma}_j \in \{2,\dots,L\}^m$) such that $|\bm{\sigma}| := \sum_{i=1}^j |\bm{\sigma}_i| \leq 2(L-1)$;
        \item $(DF \circ F^{-1})^{-L-j}$ is the inverse of $DF$, evaluated at $F^{-1}(\, \cdot \,)$, composed with (i.e., multiplied by) itself $L+j$ times;
        \item $\mathbf{v}_{\bm{\alpha},j}(\bm{\sigma},\mathbf{q}) \in \R^m$ are vectors being multiplied by $(DF \circ F^{-1})^{-L-j}$; and
        \item $\bm{\sigma}_i$ is the $i$\textsuperscript{th} row of $\bm{\sigma}$ and $F_{q_i}$ is the $q_i$\textsuperscript{th} component function of $F$. 
    \end{enumerate}
\end{lem}

From the regularity condition \eqref{eq:strong-reg} and the norm bound in \eqref{eq:diffeo-deriv-bounds}, we conclude that
\begin{align}
    \big| \partial_\mathbf{u}^{\bm{\alpha}} F_{\bm{\lambda}''}^{-1} \big| &= \left| \sum_{k=0}^{L-1} \sum_{j=1}^{L-1} \sum_{\bm{\sigma},\mathbf{q}} \big( D_{\bm{\lambda}} F_{\bm{\lambda}''} \circ F_{\bm{\lambda}''}^{-1} \big)^{-L-j} \mathbf{v}_{\bm{\alpha},j}(\bm{\sigma},\mathbf{q}) \prod_{i=1}^j \big( \partial^{\bm{\sigma}_i} (F_{\bm{\lambda}''})_{q_i} \big) \circ F_{\bm{\lambda}''}^{-1} \right| \nonumber \\
    &\leq \sum_{k=0}^{L-1} \sum_{j=1}^{L-1} \sum_{\bm{\sigma},\mathbf{q}} \big\| D_{\bm{\lambda}} F_{\bm{\lambda}''} \circ F_{\bm{\lambda}''}^{-1} \big\|^{-L-j} |\mathbf{v}_{\bm{\alpha},j}(\bm{\sigma},\mathbf{q})| \prod_{i=1}^j \big| \big( \partial^{\bm{\sigma}_i} (F_{\bm{\lambda}''})_{q_i} \big) \circ F_{\bm{\lambda}''}^{-1} \big| \nonumber \\
    &\lesssim_{\beta,L} \sum_{k=0}^{L-1} \sum_{j=1}^{L-1} \sum_{\bm{\sigma},\mathbf{q}} r^{-\beta(L+j)} |\mathbf{v}_{\bm{\alpha},j}(\bm{\sigma},\mathbf{q})| \prod_{i=1}^j r^{-\beta|\bm{\sigma}_i|} \nonumber \\
    &\lesssim_L \sum_{k=0}^{L-1} \sum_{j=1}^{L-1} \sum_{\bm{\sigma},\mathbf{q}} r^{-\beta(L+j+|\bm{\sigma}|)} \lesssim_{|\Omega|,L} r^{-\beta(L+(L-1)+2(L-1))} \nonumber \\
    &= r^{-\beta(4L-3)} \lesssim_{|\Omega|,\beta} r^{-4\beta L}. \label{eq:F-inv-est}
\end{align}
This will allow us to bound the derivatives of $g_{\bm{\lambda}''}$, all of whose factors are functions of $F_{\bm{\lambda}''}^{-1}(\mathbf{u})$ alone, but we will also need a multivariable chain rule to carry out this computation. We introduce this second differentiation formula now, starting with some notation.

Let $\prec$ denote the lexicographic order on the set of $N$-tuples over $\N$. For $\bm{\alpha} \in \N^N$ and $\bm{\gamma} \in \N^M$ (where $M$ plays the role of $m+n$), define the following collection of pairs consisting of an $L$-tuple of $N$-tuples and an $L$-tuple of $M$-tuples, $L := |\bm{\alpha}|$:
\begin{align*}
    P(\bm{\alpha},\bm{\gamma}) &:= \Bigg\{ (\bm{\sigma},\bm{\tau}) \in (\N^N)^L \times (\N^M)^L \!: \ \sum_{i=1}^L \bm{\tau}_i = \bm{\gamma}, \ \sum_{i=1}^L |\bm{\tau}_i| \bm{\sigma}_i = \bm{\alpha}, \text{ and } \\
    &\hspace{0.9cm} \exists \+ 1 \leq p \leq L \text{ s.t. } \bm{\sigma}_1 = \cdots = \bm{\sigma}_{L-p} = \bm{0}, \ \bm{\tau}_1 = \cdots = \bm{\tau}_{L-p} = \bm{0}, \\
    &\hspace{0.9cm} \big| \bm{\tau}_{L-p+1} \big|, \dots, \big| \bm{\tau}_L \big| > 0, \text{ and } \bm{0} \prec \bm{\sigma}_{L-p+1} \prec \cdots \prec \bm{\sigma}_L \Bigg\}.
\end{align*}
Here $\bm{\sigma}_i \in \N^N$ is the $i$\textsuperscript{th} component of $\bm{\sigma}$ for $i = 1,\dots,L$ (and similarly for $\bm{\tau}_i \in \N^M$). For each $(\bm{\sigma},\bm{\tau}) \in P(\bm{\alpha},\bm{\gamma})$ we denote
\begin{equation*}
    b_{\bm{\sigma},\bm{\tau}} := \bm{\alpha}! \prod_{i=1}^L \frac{1}{\bm{\tau}_i! (\bm{\sigma}_i!)^{|\bm{\tau}_i|}},
\end{equation*}
where $\bm{\alpha}! := \alpha_1 \cdots \alpha_N$. The only important points are that there $\lesssim_L 1$ coefficients $b_{\bm{\sigma},\bm{\tau}}$, all of which are positive and $\lesssim_L 1$, and that the pairs $(\bm{\sigma},\bm{\tau}) \in P(\bm{\alpha},\bm{\gamma})$ satisfy certain arithmetic conditions.

\vspace*{-0.035cm}

\begin{lem}[Multivariate Fa\`{a} di Bruno formula \cite{constantine1996multivariate}] \label{thm:faa-di-bruno}
    Let $N,M,L \in \Z_+$. If $V \subseteq \R^N$ is an open set, $f \in C^L(V;\R)$, $G \in C^L(\R^N;\R^M)$, and $\bm{\alpha} \in \N^N$ with $|\bm{\alpha}| \leq L$, then 
    \begin{equation*}
        \partial_{\bm{\lambda}}^{\bm{\alpha}} (f \circ G)(\bm{\lambda}) = \sum_{1 \leq |\bm{\gamma}| \leq |\bm{\alpha}|} \partial_{\mathbf{z}}^{\bm{\gamma}} f(\mathbf{z}) \big|_{\mathbf{z}=G(\bm{\lambda})} \sum_{\bm{\sigma},\bm{\tau}} b_{\bm{\sigma},\bm{\tau}} \prod_{i=1}^{|\bm{\alpha}|} \big[ \partial_{\bm{\lambda}}^{\bm{\tau}_i} G(\bm{\lambda}) \big]^{\bm{\sigma}_i}
    \end{equation*}
    for all $\bm{\lambda} \in V$, where the second sum runs over all $(\bm{\sigma},\bm{\tau}) \in P(\bm{\alpha},\bm{\gamma})$.
\end{lem}

\vspace*{-0.035cm}

Here, $[\mathbf{z}]^{\bm{\rho}} := z_1^{\rho_1} \cdots z_M^{\rho_M}$. We continue our estimation of $g_{\bm{\lambda}''}$ and its derivatives by treating each factor in \eqref{eq:g-lambda-double-prime} individually. First take $f = \rho_1$ (so $m$ plays the role of $N$) and $G = F_{\bm{\lambda}''}^{-1}$ (so $M = m$ as well). There are $\sim_{|\bm{\alpha}|} 1$ terms in the outermost sum in the Fa\`{a} di Bruno formula and $\big| \partial_{\bm{\lambda}'}^{\bm{\gamma}} \rho_1(\bm{\lambda}') \big| \lesssim_{K,\rho,|\bm{\gamma}|} 1$ for all $\bm{\gamma}$ and $\bm{\lambda}'$, so \vspace*{0.1cm}
\begin{align*}
    & \big| \partial_\mathbf{u}^{\bm{\alpha}} \big( \rho_1 \circ F_{\bm{\lambda}''}^{-1} \big)(\mathbf{u}) \big| \\[-0.2cm]
    &\qquad \leq \sum_{1 \leq |\bm{\gamma}| \leq |\bm{\alpha}|} \big| \partial_{\bm{\lambda}'}^{\bm{\gamma}} \rho_1(\bm{\lambda}') \big| \Big|_{\bm{\lambda}' = F_{\bm{\lambda}''}^{-1}(\mathbf{u})} \sum_{\bm{\sigma},\bm{\tau}} b_{\bm{\sigma},\bm{\tau}} \prod_{i=1}^{|\bm{\alpha}|} \Big| \big[ \partial_\mathbf{u}^{\bm{\tau}_i} F_{\bm{\lambda}''}^{-1}(\mathbf{u}) \big]^{\bm{\sigma}_i} \Big| \\[-0.1cm]
    &\qquad \lesssim_{|\bm{\alpha}|} \sum_{\bm{\sigma},\bm{\tau}} b_{\bm{\sigma},\bm{\tau}} \prod_{i=1}^{|\bm{\alpha}|} \Big| \big[ \partial_\mathbf{u}^{\bm{\tau}_i} F_{\bm{\lambda}''}^{-1}(\mathbf{u}) \big]^{\bm{\sigma}_i} \Big| \\[-0.1cm]
    &\qquad \lesssim_{|\bm{\alpha}|} \sum_{\bm{\sigma},\bm{\tau}} b_{\bm{\sigma},\bm{\tau}} \prod_{i=1}^{|\bm{\alpha}|} r^{-4\beta|\bm{\tau}_i|\bm{\sigma}_i} \\[-0.1cm]
    &\qquad = \sum_{\bm{\sigma},\bm{\tau}} b_{\bm{\sigma},\bm{\tau}} \+ r^{-4\beta|\bm{\alpha}|} \lesssim_{|\bm{\alpha}|} r^{-4\beta|\bm{\alpha}|}.
\end{align*}
This same computation and the derivative bounds in \eqref{eq:chi-k} likewise imply
\begin{equation*}
    \big| \partial_\mathbf{u}^{\bm{\alpha}} \big( \chi_k \circ F_{\bm{\lambda}''}^{-1} \big)(\mathbf{u}) \big| \lesssim r^{-(m+6)\beta|\bm{\alpha}|}.
\end{equation*}
Before treating $h_{\bm{\lambda}''}$ and its derivatives, we skip ahead to the Jacobian determinant factor. Since $F_{\bm{\lambda}''}$ is a diffeomorphism, assume without loss of generality that $J_{\bm{\lambda}''}(\mathbf{u}) := \det D_\mathbf{u} F_{\bm{\lambda}''}^{-1}(\mathbf{u})$ for all $\mathbf{u} \in B_k'$. Then $J_{\bm{\lambda}''}$ is a sum of $m^2$ terms of $m$ factors $\partial_{u_j} F_{\bm{\lambda}'',\ell}^{-1}(\mathbf{u})$, the $\ell$\textsuperscript{th} component of $\partial_{u_j} F_{\bm{\lambda}''}^{-1}(\mathbf{u})$. Consequently, $\partial_\mathbf{u}^{\bm{\alpha}} J_{\bm{\lambda}''}(\mathbf{u})$ is a sum of $m$-fold products of factors $\partial_\mathbf{u}^{\bm{\gamma}} F_{\bm{\lambda}'',\ell}^{-1}(\mathbf{u})$ with $1 \leq |\bm{\gamma}| \leq |\bm{\alpha}|$, and the number of terms in the sum is a function of $m$ and $|\bm{\alpha}|$ alone. Since
\begin{equation*}
    \big| \partial_\mathbf{u}^{\bm{\gamma}} F_{\bm{\lambda}''}^{-1}(\mathbf{u}) \big| \lesssim_{\beta,|\bm{\gamma}|} r^{-\beta(4|\bm{\gamma}|-3)} \sim_{|\Omega|,\beta,|\bm{\alpha}|} r^{-4\beta|\bm{\alpha}|},
\end{equation*}
the crude bound
\begin{equation*}
    \big| \partial_\mathbf{u}^{\bm{\alpha}} J_{\bm{\lambda}''}(\mathbf{u}) \big| \lesssim_{m,|\bm{\alpha}|} \big( r^{-4\beta|\bm{\alpha}|} \big)^m = r^{-4m\beta|\bm{\alpha}|} \vspace*{-0.025cm}
\end{equation*}
follows.

Now we double back to the factor $h_{\bm{\lambda}''}$ that we skipped. To control this, we again apply Theorem \ref{thm:faa-di-bruno}, now with $f(\bm{\lambda}') = \psi\big( 2^{j-i} r D_{(\bm{\lambda}',\bm{\lambda}'')} \Phi_{(\bm{\lambda}',\bm{\lambda}'')} \big)$ and $G(\mathbf{u}) = F_{\bm{\lambda}''}^{-1}$:
\begin{equation} \label{eq:h-lambda-double-prime}
\begin{aligned}
    \big| \partial_\mathbf{u}^{\bm{\alpha}} h_{\bm{\lambda}''}(\mathbf{u}) \big| &\leq \sum_{1 \leq |\bm{\gamma}| \leq |\bm{\alpha}|} \big| \partial_{\bm{\lambda}'}^{\bm{\gamma}} \psi\big( 2^{j-i} r D_{(\bm{\lambda}',\bm{\lambda}'')} \Phi_{(\bm{\lambda}',\bm{\lambda}'')} \big) \big| \Big|_{\bm{\lambda}' = F_{\bm{\lambda}''}^{-1}(\mathbf{u})} \\
    &\hspace*{2.25cm} \cdot \sum_{\bm{\sigma},\bm{\tau}} b_{\bm{\sigma},\bm{\gamma}} \prod_{k=1}^{|\bm{\alpha}|} \Big| \big[ \partial_\mathbf{u}^{\bm{\tau}_k} F_{\bm{\lambda}''}^{-1}(\mathbf{u}) \big]^{\bm{\sigma}_k} \Big|.
\end{aligned}
\end{equation}
The expression in the second line is amenable to \eqref{eq:F-inv-est}:
\begin{equation*}
    \sum_{\bm{\sigma},\bm{\tau}} b_{\bm{\sigma},\bm{\gamma}} \prod_{k=1}^{|\bm{\alpha}|} \Big| \big[ \partial_\mathbf{u}^{\bm{\tau}_k} F_{\bm{\lambda}''}^{-1}(\mathbf{u}) \big]^{\bm{\sigma}_k} \Big| \lesssim_{|\bm{\alpha}|} \prod_{k=1}^{|\bm{\alpha}|} r^{-4|\bm{\tau}_k|\bm{\sigma}_k} \lesssim_{|\Omega|} r^{-4\beta|\bm{\alpha}|}.
\end{equation*}
To bound the expression in the first line, recall that $\psi$ is a rapidly-decreasing function of $mN$ variables, so $\big| \partial_\mathbf{t}^{\bm{\gamma}} \psi(\mathbf{t}) \big| \lesssim_{p,|\bm{\alpha}|} |\mathbf{t}|^{-p}$ for $p \in \N$ and all $1 \leq |\bm{\gamma}| \leq |\bm{\alpha}|$. In particular, appealing to Theorem \ref{thm:faa-di-bruno} once more and applying this rapid decay bound with $p = |\bm{\alpha}|$ yields
\begin{align*}
    & \big| \partial_{\bm{\lambda}'}^{\bm{\alpha}} \psi\big( 2^{j-i} r D_{(\bm{\lambda}',\bm{\lambda}'')} \Phi_{(\bm{\lambda}',\bm{\lambda}'')} \big) \big| \\
    &\quad \leq \sum_{1 \leq |\bm{\gamma}| \leq |\bm{\alpha}|} \big| \partial_\mathbf{t}^{\bm{\gamma}} \psi(\mathbf{t}) \big| \Big|_{\mathbf{t} = 2^{j-i} r D_{\bm{\lambda}} \Phi_{\bm{\lambda}}} \sum_{\bm{\sigma},\bm{\tau}} b_{\bm{\sigma},\bm{\gamma}} \prod_{k=1}^{|\bm{\alpha}|} \Big| \big[ 2^{j-i} r \+ \partial_{\bm{\lambda}'}^{\bm{\tau}_k} D_{\bm{\lambda}'} \Phi_{(\bm{\lambda}',\bm{\lambda}'')} \big]^{\bm{\sigma}_k} \Big| \\
    &\quad \lesssim_{\beta,|\bm{\alpha}|} \sum_{1 \leq |\bm{\gamma}| \leq |\bm{\alpha}|} \big\| 2^{j-i} r D_{\bm{\lambda}} \Phi_{\bm{\lambda}} \big\|^{-|\bm{\alpha}|} \sum_{\bm{\sigma},\bm{\tau}} b_{\bm{\sigma},\bm{\gamma}} \prod_{k=1}^{|\bm{\alpha}|} (2^{j-i}r)^{\bm{\sigma}_k} r^{-\beta(|\bm{\tau}_k|+1)\bm{\sigma}_k} \\
    &\quad \lesssim_{|\Omega|} \sum_{1 \leq |\bm{\gamma}| \leq |\bm{\alpha}|} (2^{j-i}r)^{-|\bm{\alpha}|} \sum_{\bm{\sigma},\bm{\tau}} b_{\bm{\sigma},\bm{\gamma}} \+ (2^{j-i}r)^{|\bm{\alpha}|} \cdot r^{-2\beta|\bm{\alpha}|} \\
    &\quad \lesssim_{|\bm{\alpha}|} r^{-2\beta|\bm{\alpha}|}.
\end{align*}
Thus, by \eqref{eq:h-lambda-double-prime}
\begin{equation*}
    \big| \partial_\mathbf{u}^{\bm{\alpha}} h_{\bm{\lambda}''}(\mathbf{u}) \big| \lesssim \sum_{1 \leq |\bm{\gamma}| \leq |\bm{\alpha}|} r^{-2\beta|\bm{\alpha}|} \cdot r^{-4\beta|\bm{\alpha}|} \sim_{|\bm{\alpha}|} r^{-6\beta|\bm{\alpha}|}.
\end{equation*}
This in hand, we can at last bound the derivatives of $g_{\bm{\lambda}''}$. The number of terms that arise by application of the product rule to \eqref{eq:g-lambda-double-prime} is a constant depending only on $|\bm{\alpha}|$, so we may combine the bounds on the derivatives of each of the four factors to conclude that
\begin{equation} \label{eq:g-lambda-double-prime-derivs}
    \big| \partial_\mathbf{u}^{\bm{\alpha}} g_{\bm{\lambda}''}(\mathbf{u}) \big| \lesssim_{|\bm{\alpha}|} r^{-4\beta|\bm{\alpha}|} \cdot r^{-(m+6)\beta|\bm{\alpha}|} \cdot r^{-6\beta|\bm{\alpha}|} \cdot r^{-4m\beta|\bm{\alpha}|} = r^{-(5m+16)\beta|\bm{\alpha}|}.
\end{equation}
This is valid for all $|\bm{\alpha}| \geq 1$. In the $\bm{\alpha} = \mathbf{0}$ case, it is enough to note that $\rho_1 \circ F_{\bm{\alpha}''}^{-1}$, $\chi_k \circ F_{\bm{\lambda}''}^{-1}$, and $h_{\bm{\lambda}''}$ are bounded independently of $\bm{\lambda}''$, so $\big| g_{\bm{\lambda}''}(\mathbf{u}) \big| \lesssim 1 \cdot 1 \cdot 1 \cdot r^{-\beta} = r^{-\beta}$ for all $\mathbf{u} \in B_k'$ by \eqref{eq:diffeo-deriv-bounds}.

The final order of business is to write out the Taylor series approximation of $g_{\bm{\lambda}''}$ and use it to bound \eqref{eq:final-integral}. (Recall our purpose is to use this to control \eqref{eq:app-line-4}, which is a sum of terms of this form, and thereby prove the lemma.) The factor $\rho_2(\bm{\lambda}'')$ will pose no issue, so we set it aside for now, and the other factor $\phi\big( 2^j r F_{\bm{\lambda}''}^{-1}(\mathbf{u}) )$ will take a Schwartz class bound. With $M := 2 \lceil c \rceil - m$,
\begin{align*}
    &\int_{\R^m} g_{\bm{\lambda}''}(\mathbf{u}) \phi(2^j r\mathbf{u}) \, d\mathbf{u} \\
    &\hspace{1.5cm} = \int_{\{ |\mathbf{u}| \leq (2^j r)^{-1/2} \}} \phi(2^j r\mathbf{u}) \Bigg( \sum_{0 \leq |\bm{\alpha}| \leq M} \frac{\partial_\mathbf{u}^{\bm{\alpha}} g_{\bm{\lambda}''}(\mathbf{0})}{\bm{\alpha}!} \mathbf{u}^{\bm{\alpha}} \\
    &\hspace{2.5cm} + \sum_{|\bm{\alpha}| = M+1} \int_0^1 \big[ \partial_\mathbf{u}^{\bm{\alpha}} g_{\bm{\lambda}''}(t\mathbf{u}) - \partial_\mathbf{u}^{\bm{\alpha}} g_{\bm{\lambda}''}(\mathbf{0}) \big] (1-t)^M \+ dt \, \frac{\mathbf{u}^{\bm{\alpha}}}{\bm{\alpha}!} \Bigg) d\mathbf{u} \\
    &\hspace{2cm} + \int_{\{ |\mathbf{u}| > (2^j r)^{-1/2} \}} \phi(2^j r\mathbf{u}) g_{\bm{\lambda}''}(\mathbf{u}) \, d\mathbf{u} \\
    &\hspace{1.5cm} =: A_1 + A_2.
\end{align*}
To complete the proof, we must estimate $A_1$ and $A_2$ uniformly in $\bm{\lambda}''$. The latter is simple owing the rapid decay of $\phi$ and the basic inequality $\big| g_{\bm{\lambda}''}(\mathbf{u}) \big| \lesssim r^{-\beta}$:
\begin{align*}
    |A_2| &\lesssim r^{-\beta} \int_{\{ |\mathbf{u}| > (2^j r)^{-1/2} \}} |\phi(2^j r\mathbf{u})| \, d\mathbf{u} \\
    &\lesssim_c r^{-\beta} \int_{\{ |\mathbf{u}| > (2^j r)^{-1/2} \}} (2^j r |\mathbf{u}|)^{-2c-2m} \, d\mathbf{u} \\
    &\sim r^{-\beta} \cdot (2^j r)^{-2c-2m} \cdot (2^j r)^{c+2m} = r^{-\beta} (2^j r)^{-c}.
\end{align*}
To estimate $A_1$, observe first that
\begin{align*}
    \int_{\R^m} \mathbf{u}^{\bm{\alpha}} \phi(2^j r\mathbf{u}) \, d\mathbf{u} &= \mathcal{F}_\mathbf{u}\big[ \mathbf{u}^{\bm{\alpha}} \phi(2^j r\mathbf{u}) \big](\mathbf{0}) \\
    &= (-2\pi i)^{|\bm{\alpha}|} \partial_{\bm{\xi}}^{\bm{\alpha}} \mathcal{F}_\mathbf{u}\big[ \phi(2^j r\mathbf{u}) \big](\bm{\xi}) \big|_{\bm{\xi} = \mathbf{0}} \\
    &= (-2\pi i)^{|\bm{\alpha}|} (2^j r)^{-m} \partial_{\bm{\xi}}^{\bm{\alpha}} \medhat{\phi}(2^{-j}r^{-1}\bm{\xi}) \big|_{\bm{\xi} = \mathbf{0}} \\
    &= (-2\pi i)^{|\bm{\alpha}|} (2^j r)^{-m} \partial^{\bm{\alpha}} \eta^{(1)}(\mathbf{0}),
\end{align*}
and this vanishes for all $\bm{\alpha} \in \N^m$ because $\eta$ is identically $0$ in a neighborhood of the origin. Therefore,
\begin{align*}
    &\int_{\{ |\mathbf{u}| \leq (2^j r)^{-1/2} \}} \phi(2^j r\mathbf{u}) \sum_{0 \leq |\bm{\alpha}| \leq M} \frac{\partial_\mathbf{u}^{\bm{\alpha}} g_{\bm{\lambda}''}(\mathbf{0})}{\bm{\alpha}!} \mathbf{u}^{\bm{\alpha}} \, d\mathbf{u} \\
    &\hspace{0.5cm} = -\int_{\{ |\mathbf{u}| > (2^j r)^{-1/2} \}} \phi(2^j r\mathbf{u}) \sum_{0 \leq |\bm{\alpha}| \leq M} \frac{\partial_\mathbf{u}^{\bm{\alpha}} g_{\bm{\lambda}''}(\mathbf{0})}{\bm{\alpha}!} \mathbf{u}^{\bm{\alpha}} \, d\mathbf{u}.
\end{align*}
We already know that the $\bm{\alpha} = \mathbf{0}$ term is $\lesssim_c r^{-\beta}$ in absolute value. To bound the remaining terms, recall that we assumed $2^j r^{1 + a_1' \beta} > 1$. At last setting $a_1' := 5m+16$, we use \eqref{eq:g-lambda-double-prime-derivs} and the rapid decay bound $|\phi(2^j r\mathbf{u})| \lesssim_M (2^j r|\mathbf{u}|)^{-2c-m-|\bm{\alpha}|}$ to conclude that
\begin{align*}
    &\left| \int_{\{ |\mathbf{u}| > (2^j r)^{-1/2} \}} \phi(2^j r\mathbf{u}) \sum_{1 \leq |\bm{\alpha}| \leq M} \frac{\partial_\mathbf{u}^{\bm{\alpha}} g_{\bm{\lambda}''}(\mathbf{0})}{\bm{\alpha}!} \mathbf{u}^{\bm{\alpha}} \, d\mathbf{u} \right| \\
    &\qquad \lesssim_M \sum_{1 \leq |\bm{\alpha}| \leq M} \int_{\{ |\mathbf{u}| > (2^j r)^{-1/2} \}} (2^j r|\mathbf{u}|)^{-2c-m-|\bm{\alpha}|} \cdot r^{-(5m+16)\beta|\bm{\alpha}|} \cdot |\mathbf{u}|^{|\bm{\alpha}|} \, d\mathbf{u} \\
    &\qquad = \sum_{1 \leq |\bm{\alpha}| \leq M} (2^j r)^{-2c-m-|\bm{\alpha}|} \cdot r^{-(5m+16)\beta|\bm{\alpha}|} \int_{\{ |\mathbf{u}| > (2^j r)^{-1/2} \}} |\mathbf{u}|^{-2c-m} \, d\mathbf{u} \\
    &\qquad \sim_{m,c} \sum_{1 \leq |\bm{\alpha}| \leq M} (2^j r)^{-2c-m-|\bm{\alpha}|} \cdot r^{-(5m+16)\beta|\bm{\alpha}|} \cdot (2^j r)^c \\
    &\qquad = \sum_{1 \leq |\bm{\alpha}| \leq M} (2^j r)^{-c-m-|\bm{\alpha}|} \cdot r^{-(5m+16)\beta|\bm{\alpha}|} \\
    &\qquad \leq \sum_{1 \leq |\bm{\alpha}| \leq M} r^{-m(5m+16)\beta} (2^j r)^{-c} \\
    &\qquad \sim_c r^{-m(5m+16)\beta} (2^j r)^{-c}.
\end{align*}
It remains only to bound the Taylor error term. By Taylor's theorem and the boundedness of $\phi$,
\begin{align*}
    & \left| \int_{\{ |\mathbf{u}| \leq (2^j r)^{-1/2} \}} \phi(2^j r\mathbf{u}) \sum_{|\bm{\alpha}| = M+1} \int_0^1 \big[ \partial_\mathbf{u}^{\bm{\alpha}} g_{\bm{\lambda}''}(t\mathbf{u}) - \partial_\mathbf{u}^{\bm{\alpha}} g_{\bm{\lambda}''}(\mathbf{0}) \big] (1-t)^M \+ dt \, \frac{\mathbf{u}^{\bm{\alpha}}}{\bm{\alpha}!} d\mathbf{u} \right| \\
    &\qquad \leq \int_{\{ |\mathbf{u}| \leq (2^j r)^{-1/2} \}} O\!\left( \sup_{|\bm{\alpha}|=M+1} \big\| \partial_\mathbf{u}^{\bm{\alpha}} g_{\bm{\lambda}''} \big\|_{L^\infty(B_k')} |\mathbf{u}|^{M+1} \right) \! d\mathbf{u} \\
    &\qquad \lesssim r^{-(5m+16)\beta(M+1)} \int_{\{ |\mathbf{u}| \leq (2^j r)^{-1/2} \}} |\mathbf{u}|^{M+1} \, d\mathbf{u} \\
    &\qquad \sim_c r^{-2c(5m+16)\beta} \cdot (2^j r)^{-(2\lceil c \rceil + 1)/2} \\
    &\qquad \lesssim_c r^{-2c(5m+16)\beta} \cdot (2^j r)^{-c}.
\end{align*}
This shows that
\begin{align*}
    |A_1| + |A_2| &\lesssim r^{-m(5m+16)\beta} (2^j r)^{-c} + r^{-2c(5m+16)\beta} \cdot (2^j r)^{-c} + r^{-\beta} (2^j r)^{-c} \\
    &\lesssim_{|\Omega|,\beta,c} r^{-a_1 c\beta} \cdot (2^j r)^{-c} \sim \big( 1 + 2^j r^{1 + a_1 \beta} \big)^{-c},
\end{align*}
where $a_1 := \max \{ m,2 \} (5m+16)$. All implicit constants are independent of $\bm{\lambda}''$ and $\rho_2 \in L^1(\R^{N-m})$, so
\begin{equation*}
    \int_{\R^{N-m}} \rho_2(\bm{\lambda}'') \int_{\R^m} g_{\bm{\lambda}''}(\mathbf{u}) \phi(2^j r\mathbf{u}) \, d\mathbf{u} \+ d\bm{\lambda}'' \lesssim \big( 1 + 2^j r^{1 + a_1 \beta} \big)^{-c}
\end{equation*}
as we sought to show.
\end{proof}

\phantomsection
\section*{Acknowledgments}
The author thanks Paige Bright, Caleb Marshall, and Bobby Wilson for their helpful comments on an earlier draft of this work. In addition, he is grateful to Tuomas Orponen for drawing his attention to Hovila's results on the Heisenberg group.

\bibliographystyle{plain}
\bibliography{references}

\end{document}